			\documentclass[reqno, letterpaper,11pt]{amsart}

			\usepackage{amsmath, amssymb}
			\usepackage{graphicx}
			\usepackage[all, knot, frame]{xy}
			\usepackage{tikz}
			\usetikzlibrary{decorations}
			\usetikzlibrary{arrows}
			\usepackage{setspace}
			\xyoption{arc} 
			\usepackage{fancyhdr} 
			\usepackage{latexsym} 
			\usepackage{mathrsfs} 
			\usepackage{verbatim} 
			\usepackage{amssymb} 
			\usepackage{amsmath} 
			\usepackage{enumerate} 
			\usepackage{hyperref}
	
				\def\s{\sigma }

				\def\Z{{\mathbb Z}}
				\def\Q{{\mathbb Q}}
				
				\def\C{{\mathbb C}}

				\def\cal{\mathcal }
		
				\def\l{\left }
				\def\r{\right }
				\def\<{\l\langle}
				\def\>{\r\rangle}
		
				\def\Poincare{Poincar\'e }

				\def\tb{\mbox{tb}}
				\def\lk{\ell k}

				\newtheorem{thm}{Theorem}[section]
				\newtheorem{cor}[thm]{Corollary}
				\newtheorem{proposition}[thm]{Proposition}
				\newtheorem{lem}[thm]{Lemma}
		
				\theoremstyle{definition}
				\newtheorem{definition}[thm]{Definition}
				\newtheorem{notation}[thm]{Notation}
				\newtheorem{remark}[thm]{Remark}

	\title{On Handlebody Structures of Rational Balls}
	\author{Luke Williams}
	\address{Mathematics Department, Michigan State University, East Lansing, MI 48824}
	\email{will2086@math.msu.edu}

\begin{document}
	\begin{abstract}
		It is known that for coprime integers $p>q\geq 1$, the lens space $L(p^2,pq-1)$ bounds 
		a rational ball, $B_{p,q}$, arising as the 2-fold branched cover of a (smooth)
		slice disk in $B^4$ bounding the associated 2-bridge knot.  
		Lekilli and Maydanskiy \cite{Lekili-Bpq} give handle decompositions
		for each $B_{p,q}$. Whereas,
		Yamada \cite{Yamada-Amn} gives an alternative definition of 
		rational balls, $A_{m,n}$, bounding $L(p^2,pq-1)$ by 
		their handlebody decompositions alone.
		We show that these two families coincide - answering a question
		of Kadokami and Yamada in \cite{Yamada-LensSurgeries}.  To that end, we show that each
		$A_{m,n}$ admits a Stein filling of the ``standard'' contact structure, 
		$\bar{\xi}_{st}$, on $L(p^2,pq-1)$ investigated by Lisca
		in \cite{Lisca-LensFillings}.  
	\end{abstract}		
	
	\maketitle
\section{Introduction}\label{Section: Introduction}
	{
		For $p>q\geq 1$ relatively prime, let $B_{p,q}$ be the 4-manifold 
		obtained by attaching a 1-handle and a single 2-handle with 
		framing $pq-1$ to $B^4$
		by wrapping the attaching circle of the 2-handle $p$-times around the 
		1-handle with a $q/p$-twist 
		(see the left side of Figure \ref{Figure: Boundary Diff}).
		From this description, it is immediate that $B_{p,q}$ is
		always a rational homology ball.  Lekili and Maydanskiy 
		\cite{Lekili-Bpq} show that each such $B_{p,q}$ 
		arises as the 2-fold branched cover of $B^4$ branched over a slice disk 
		for the (slice) 2-bridge knot associated to the fraction $-p^2/(pq-1)$.
		That is, the family $B_{p,q}$ represents
		handle decompositions of the rational balls introduced by Casson
		and Harer in \cite{CassonHarer-RationalBalls}.		
		As such, $\partial B_{p,q}\approx L(p^2,pq-1)$ - throughout 
		$\approx$ denotes when two manifolds are diffeomorphic.
		
		In a similar direction, Yamada  \cite{Yamada-Amn}
		defines a family of rational 
		balls bounding $L(p^2,pq-1)$ directly via their handle decompositions:  
		For $n, m\geq 1$ relatively prime, let $A_{m,n}$ be the 4-manifold 
		obtained by attaching a 1-handle and a single 2-handle with 
		framing $mn$ to $B^4$
		by attaching the 2-handle along a simple closed curve 
		embedded on a once-punctured torus viewed in $S^1\times S^2$
		so that the attaching circle traverses the two 
		1-handles of the torus $m$ and $n$ times 
		respectively  
		(see the right side of Figure \ref{Figure: Boundary Diff}).
		Yamada goes on to define an involutive symmetric function, $A$,
		on the set of coprime pairs of positive integers such that
		if $A(p-q,q)=(m,n)$ then $\partial A_{m,n}\approx L(p^2,pq-1)$
		(see Lemma \ref{Lemma: Defining A(p-q,q)} for a definition of $A$).
		
		Given these two constructions of rational balls with
		coincident boundaries, one arrives at a natural question
		posed by Kadokami and Yamada in \cite{Yamada-LensSurgeries} as Problem 1.9: 
		When are these two families diffeomorphic, 
		homeomorphic, or even homotopic relative to their boundaries 
		as 4-manifolds?  We provide a complete
		answer to this question by proving the following theorem.	
		\begin{thm}\label{Each Amn is Stein}
			For each pair of relatively prime positive integers, $(m,n)$, $A_{m,n}$
			carries a Stein structure, $\widetilde{J}_{m,n}$, 
			filling a contact structure contactomorphic to 
			the standard contact structure 
			$\bar{\xi}_{\mbox{st}}$ on the lens space $\partial A_{m,n}$.  
			In particular,
			each $A_{m,n}\approx B_{p,q}$ if and only if 
			$\partial A_{m,n}\approx \partial B_{p,q}$.
		\end{thm}
		\noindent The proof of Theorem \ref{Each Amn is Stein} follows by first explicitly 
		writing down a Stein structure on $A_{m,n}$ using Eliashberg and Gompf's
		\cite{Gompf-SteinHandles}
		characterization of handle decompositions of Stein domains.  Then, verifying
		that the homotopy invariants of the induced contact structures on the boundary
		agree with those of $(L(p^2,pq-1),\bar{\xi}_{\mbox{st}})$, showing that
		the two structures are homotopic as 2-plane fields.  Work of Honda's
		\cite{Honda-ClassificationI} shows that this is sufficient to conclude
		that these two contact structures are contactomorphic.  
		Lisca's classification \cite{Lisca-LensFillings} of the diffeomorphism types of
		symplectic fillings of $(L(p^2,pq-1),\bar{\xi}_{\mbox{st}})$ then 
		gives that $A_{m,n}\approx B_{p,q}$.		
		In order to successfully compare the aforementioned
		homotopy invariants, we construct boundary 
		diffeomorphisms.  These boundary diffeomorphisms
		can be extended to explicit diffeomorphisms between $B_{p,q}$ and $A_{m,n}$
		through 	the carving process introduced in \cite{Akbulut-Carving}; in fact,
		we have:
		\begin{thm}\label{Theorem: Boundary Diffeomorphism}
			Let $(m,n)=A(p-q,q)$ for some $p>q>0$ relatively prime.  Then
			there exists a diffeomorphism $f:\partial B_{p,q} \to \partial A_{m,n}$
			such that $f$ carries the belt sphere, $\mu_1$, of the single 2-handle
			in $B_{p,q}$ to an unknot in $\partial A_{m,n}$ 
			(see Figure \ref{Figure: Boundary Diff}).  
			Moreover, carving $A_{m,n}$ along $f(\mu_1)$ gives $S^1\times B^3$.
			\begin{figure}[!ht]
					\centering
						\includegraphics[scale=.6]{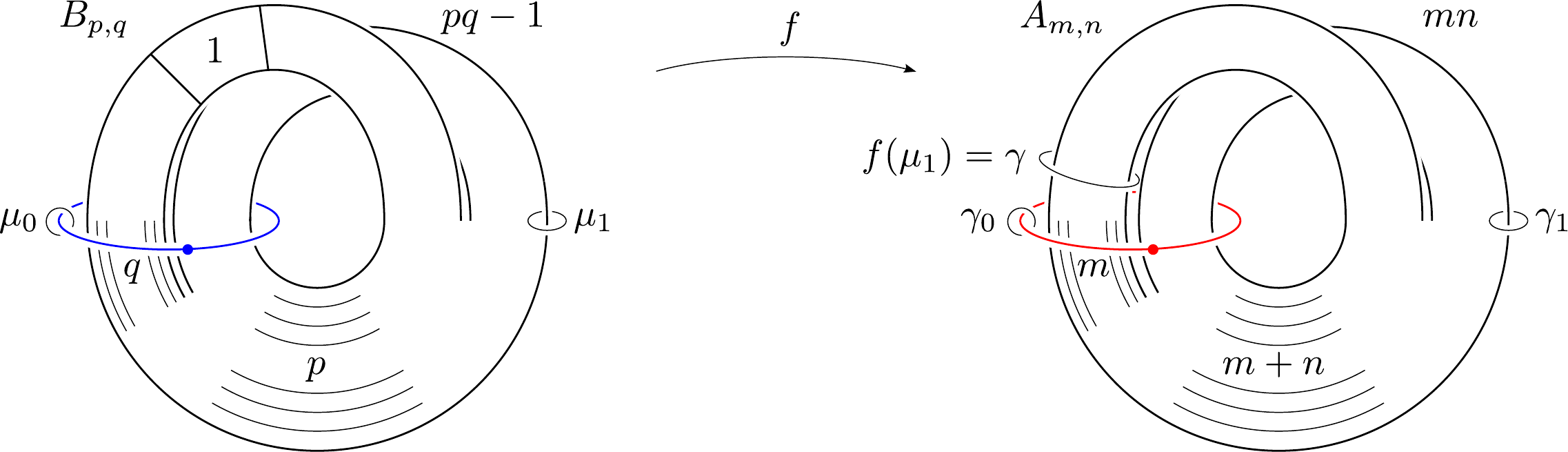}
					\caption{\small The spaces $B_{p,q}$ and $A_{m,n}$.}
						\label{Figure: Boundary Diff}
				\end{figure}				
			\end{thm}
		\begin{cor}\label{f extends}
			$f$ extends to a diffeomorphism $\tilde{f}:B_{p,q}\to A_{m,n}$.
		\end{cor}		
		
		In \cite{FintushelStern-RationalBlowdown}, Fintushel and Stern
		define a smooth operation, the rational blow-down, on 4-manifolds containing  
		certain configurations of spheres by removing a neighborhood
		of those spheres and replacing them by the rational 
		ball $B_{p,p-1}$.  In \cite{Park-GeneralizedRationalBlowdowns}, 
		Park generalized the operation to a larger set of configurations 
		at the expense of having to glue in $B_{p,q}$ for $q$ other 
		than $p-1$. In the presence of a symplectic structure (and a symplectic
		configuration of spheres), 
		both operations can be performed symplectically
		\cite{Symington-SymplecticRBD,Symington-GeneralSymplecticRBD}.  
		Moreover,
		under mild assumptions (see \cite{FintushelStern-RationalBlowdown}, 
		\cite{Park-GeneralizedRationalBlowdowns} for details), nontrivial solutions to 
		the Seiberg-Witten equations on the original 4-manifold induce 
		nontrivial solutions on the surgered manifold.  
		
		Therefore, having well understood handle decompositions for $B_{p,q}$
		allows one to construct explicit examples of rationally blown-down 
		4-manifolds.  For instance, this has been used to construct an exotic
		$\C P^2 \#6 \overline{\C P^2}$ in \cite{StipsiczSzabo-ExoticCP2connect6CP2bar}.  
		Corollary \ref{f extends} and Theorem 
		\ref{Each Amn is Stein} are then useful, since either the decomposition
		$B_{p,q}$ or $A_{m,n}$ can conceivably be used interchangeably.

		\subsection{Conventions and Assumptions}
		{
		Unless specifically stated to the contrary, throughout the paper, 
		we assume $p-q > q \geq 1$, $n>m\geq 1$, 
		and that both pairs are relatively prime.  As 
		$B_{p,q} \approx B_{p,p-q}$ and $A_{m,n}\approx A_{n,m}$,
		this assumption doesn't represent a restriction.
		We adopt the standard convention that $L(p,q)$ is the result of $-p/q$-surgery
		on the unknot in $S^3$.  It is well known that $L(p,q)$ is also given
		as the boundary of a linear plumbing of $D^2$-bundles over $S^2$ with
		Euler classes chosen according to the continued fraction associated 
		to $-p/q$:
		\[
			[c_1,\ldots, c_n] 
				\doteq 
				c_1 - \frac{1}{c_2
							-\displaystyle \frac{1}{\ddots 
									- \displaystyle \frac{1}{c_n}}}
				= -\frac{p}{q} 
		\]
		where, $c_i$ are uniquely determined provided $c_i\leq-2$.  We will 
		often forgo the uniqueness of the $c_i$'s in favor of shorter 
		continued fraction expansions and thus smaller bounding 4-manifolds.  
		
		The continued fraction associated to $-p^2/(pq-1)$
		involves the Euclidean algorithm (see \cite{CassonHarer-RationalBalls}
		as well as \cite{Yamada-Amn}).  Therefore, we use the Euclidean algorithm to 
		define sequences of remainders and divisors of $p$ and $q$ as follows:
		\begin{definition}\label{Definition: Euclidean Sequences}
			For $p>q\geq1$, relatively prime, let $\{r_i\}_{i=-1}^{\ell+2}$ and 
			$\{s_i\}_{i=0}^{\ell+1}$ 
			be defined recursively by $r_{-1}\doteq p$, $r_0\doteq q$ and
			\[
				r_{i+1} = r_{i-1} \mod r_i, \hspace{.5in} r_{i-1}=r_is_i+r_{i+1}.
			\]
			Let $\ell$ be the last index where $r_\ell>1$ so that $r_{\ell+1}=1$
			and $r_{\ell+2} \doteq 0$.
		\end{definition}
		For bookkeeping purposes, we'll differentiate between the 
		above sequences for $p$ and $q$ and
		the analogously defined sequences $\{\rho_i\}_{i=-1}^{\ell+2}$ and 
		$\{\sigma_i\}_{i=0}^{\ell+1}$ associated to $n>m\geq 1$.  
		Furthermore, provided that $p-q>q$, $\ell$ agrees between the 
		two sequences when $A(p-q,q)=(m,n)$ or $(n,m)$ 
		(see Remark \ref{Remark: A_m,n 0-framing preserved}
		and Lemma \ref{Lemma: Defining A(p-q,q)}).
		}
		
		\subsection{Organization}
		{
			The paper is organized as follows: In Section 
			\ref{Section: Stein Structures on Amn}, we construct
			Stein structures on each $A_{m,n}$ using
			Eliashberg and Gompf's characterization of handle decompositions
			of Stein domains.  In Section \ref{Section: Boundary Diffeomorphisms},
			we construct explicit boundary diffeomoprhisms from $B_{p,q}$
			and $A_{m,n}$ to their lens space boundaries - proving Theorem 
			\ref{Theorem: Boundary Diffeomorphism}. 
			In Section \ref{Section: Homotopy Invariants}, we prove
			Theorem \ref{Each Amn is Stein} by using those
			boundary diffeomorphisms to determine
			which contact structures are 
			induced by the Stein structures of
			Section \ref{Section: Stein Structures on Amn}.
			For clarity we relegate much of the required algebra to
			Section \ref{Section: Algebraic Details}. 
		}
		
	}

\section{Stein Structures on $A_{m,n}$}\label{Section: Stein Structures on Amn}
{
	In this section, we show that $A_{m,n}$ admits a Stein 
	structure.  To accomplish this, we use the handle characterization
	of Stein surfaces given in \cite{Gompf-SteinHandles}.  
	The reader can also consult \cite{Gompf-and-Stipsicz} as well
	as \cite{Ozbagci-Stipsicz} for thoughtful treatments of the subject.
	Such a Stein structure induces a (tight) contact structure on
	$\partial A_{m,n}$.  Tight contact structures on lens spaces
	are well understood; Honda completely classifies them in
	\cite{Honda-ClassificationI}.  Moreover, in \cite{Lisca-LensFillings},
	Lisca classifies the diffeomorphism types of symplectic
	fillings of $(L(p,q),\bar{\xi}_{st})$ where $\bar{\xi}_{st}$ is
	the contact structure $L(p,q)$ inherits from the unique tight
	contact structure on $S^3$ via the cyclic group action.
	In particular, Lisca defines 4-manifolds $W_{p,q}({\bf n})$,
	such that
	\begin{thm}[\cite{Lisca-LensFillings}, Theorem 1.1]
		\label{Lens Space Symplectice Filling Classification}
		Let $p>q\geq 1$ be relatively prime integers.  Then
		each symplectic filling $(W,\omega)$ of $(L(p,q),\bar{\xi}_{\mbox{st}})$ 
		is orientation preserving diffeomorphic to a smooth blowup 
		of $W_{p,q}({\bf n})$ for some ${\bf n}\in {\bf Z}_{p,q}$.
		Moreover, if $b_2(W)=0$, then $W$ is unique.    
	\end{thm}
	In light of Lisca's theorem, if we show that not only does
	$A_{m,n}$ admit a Stein structure, but that such a structure
	gives a symplectic filling of $(L(p^2,pq-1),\bar{\xi}_{st})$,
	then we immediately have that $A_{m,n}\approx B_{p,q}$ since
	it is known that $B_{p,q}$ admits a Stein structure
	giving such a filling. Indeed, by sliding the 
	2-handle of $B_{p,q}$ over the 1-handle $q$-times
	one arrives at the Stein domain, $(B_{p,q},J_{p,q})$,
	given in Figure \ref{Figure: (B_p,q,J_p,q} and investigated 
	in \cite{Lekili-Bpq}; there, the authors prove that $J_{p,q}$ fills
	the standard contact structure on $L(p^2,pq-1)$.
	\begin{proposition}\label{Proposition: A Stein structure on A_m,n}
		Each $A_{m,n}$ admits a Stein structure, $\widetilde{J}_{m,n}$, 
		specified by Figure \ref{Figure: (A_m,n,J_1)} where
		we assume $n=m\sigma_0+\rho_1$.
		\begin{figure}[!ht]
				\centering
					\includegraphics[scale=.5]{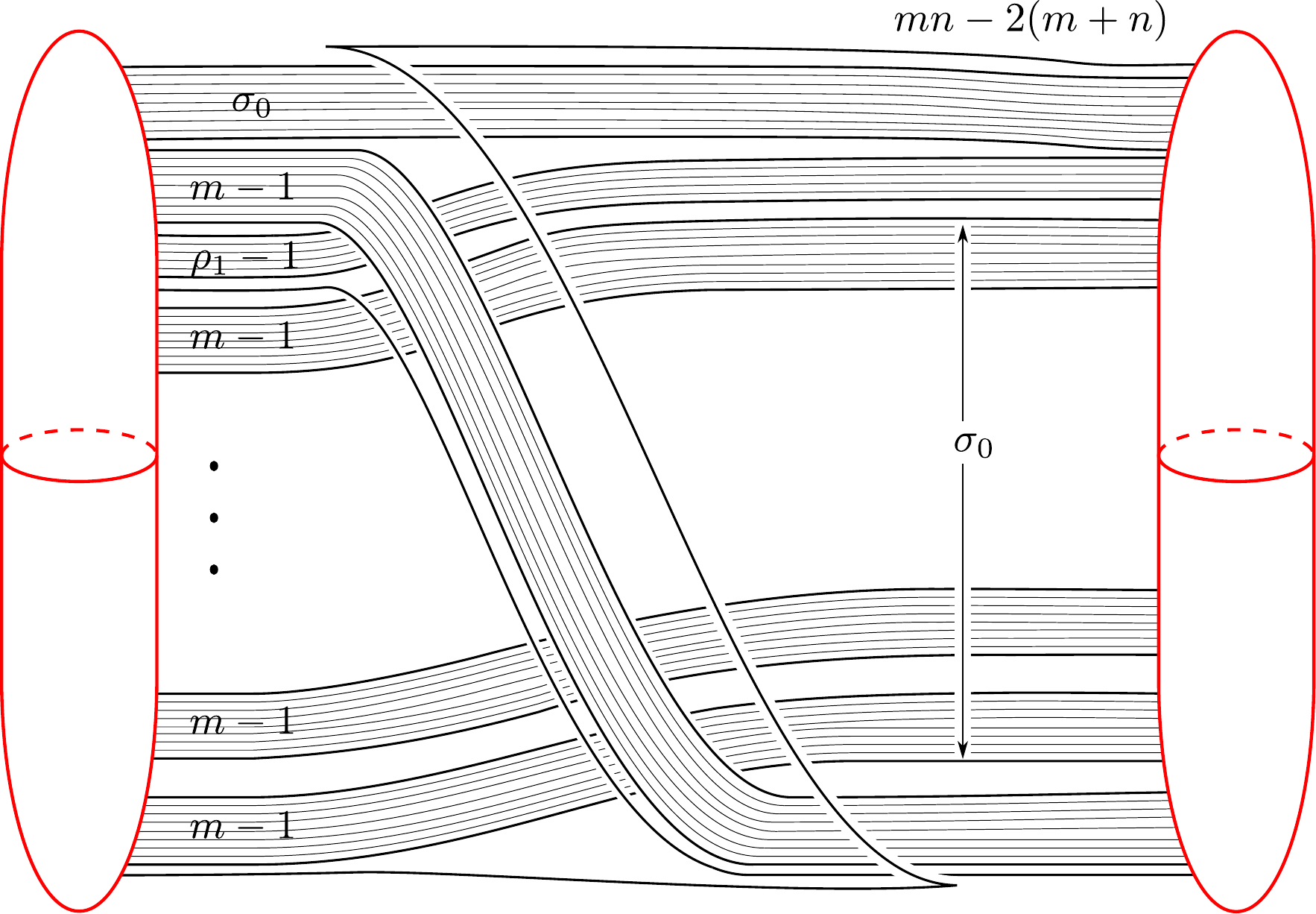}
				\caption{\small $(A_{m,n},\widetilde{J}_{m,n})$}
				\label{Figure: (A_m,n,J_1)}
		\end{figure}
	\end{proposition}
	\begin{proof}
		Notice that there are $(m-1)((\rho_1-1)+\sigma_0(m-1))+\sigma_0(m-1)$ 
		positive crossings, $m+n-1$ negative crossings and one 
		left cusp coming from the Legendrian attaching circle, 
		$K\subset S^1\times S^2$, of the 2-handle in
		Figure \ref{Figure: (A_m,n,J_1)}.  
		Then, the Thurston-Bennequin framing of $K$ is
		\begin{align*}
			\tb(K) 
				&= \mbox{\# of possitive crossings} 
					- \mbox{\# of negative crossings} 
					- \mbox{\# of left cusps}\\
				&= (m-1)(\rho_1-1)+m\sigma_0(m-1)- (m+n-1) -1\\
				&= m(m\sigma_0+\rho_1) -m -m\sigma_0 -\rho_1 -(m+n) + 1\\
				&= mn-2(m+n)+1.
		\end{align*}
		Then, \cite{Gompf-SteinHandles} gives that
		the unique Stein structure on $S^1\times B^3$ 
		extends to $A_{m,n}$ - provided
		that Figure \ref{Figure: (A_m,n,J_1)} specifies $A_{m,n}$. 
		To that end, express $A_{m,n}$ in 2-ball notation and
		slide the 2-handle once under the 1-handle (left side of 
		Figure \ref{Figure: Amn Stein 1}).
		\begin{figure}[!ht]
			\begin{tikzpicture}[xscale=.9]
				\node[inner sep=0pt] at (0,0)
    				{\includegraphics[scale=.42]{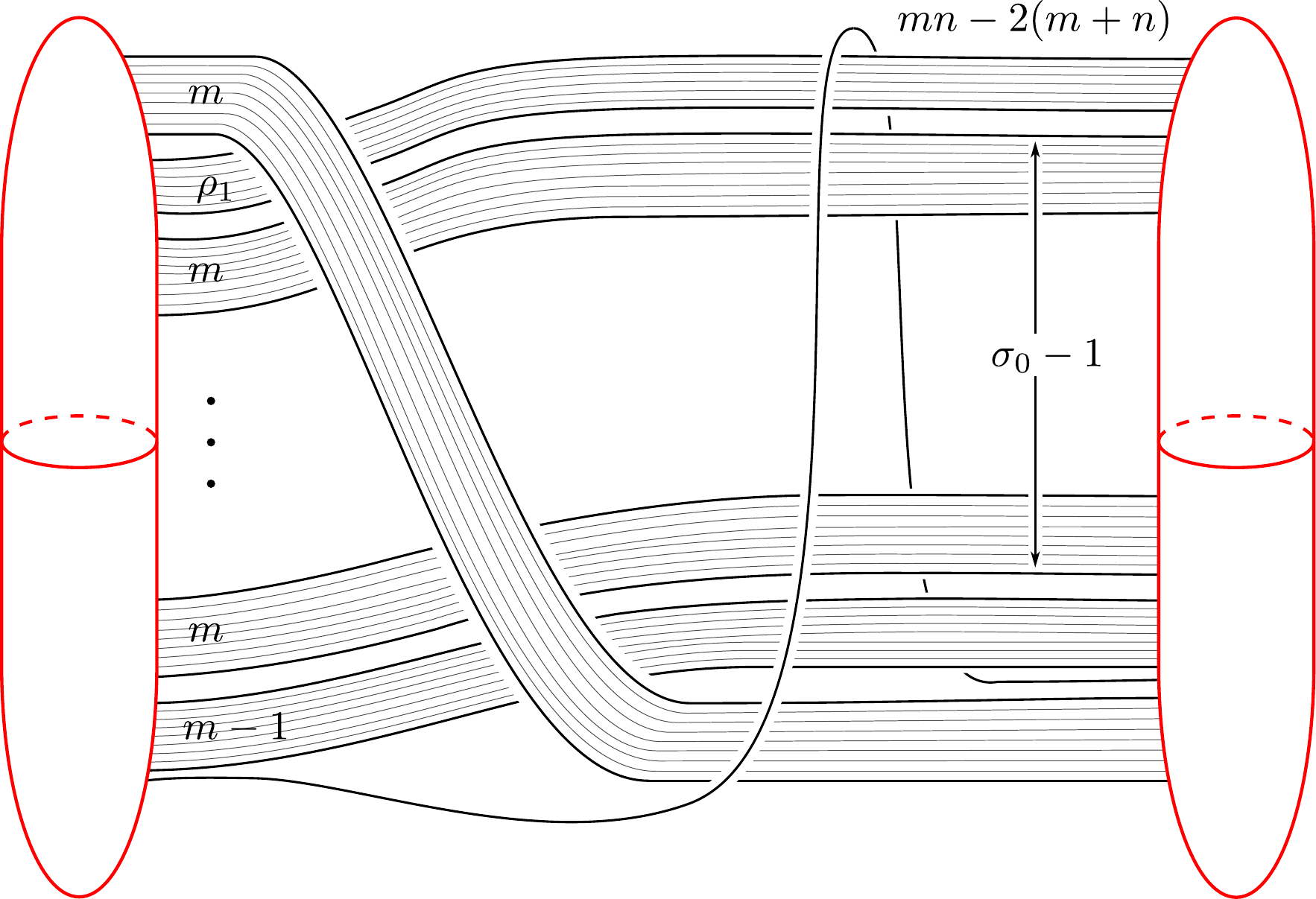}};
    			\node[inner sep=0pt] at (9,0)
					{\includegraphics[scale=.42]{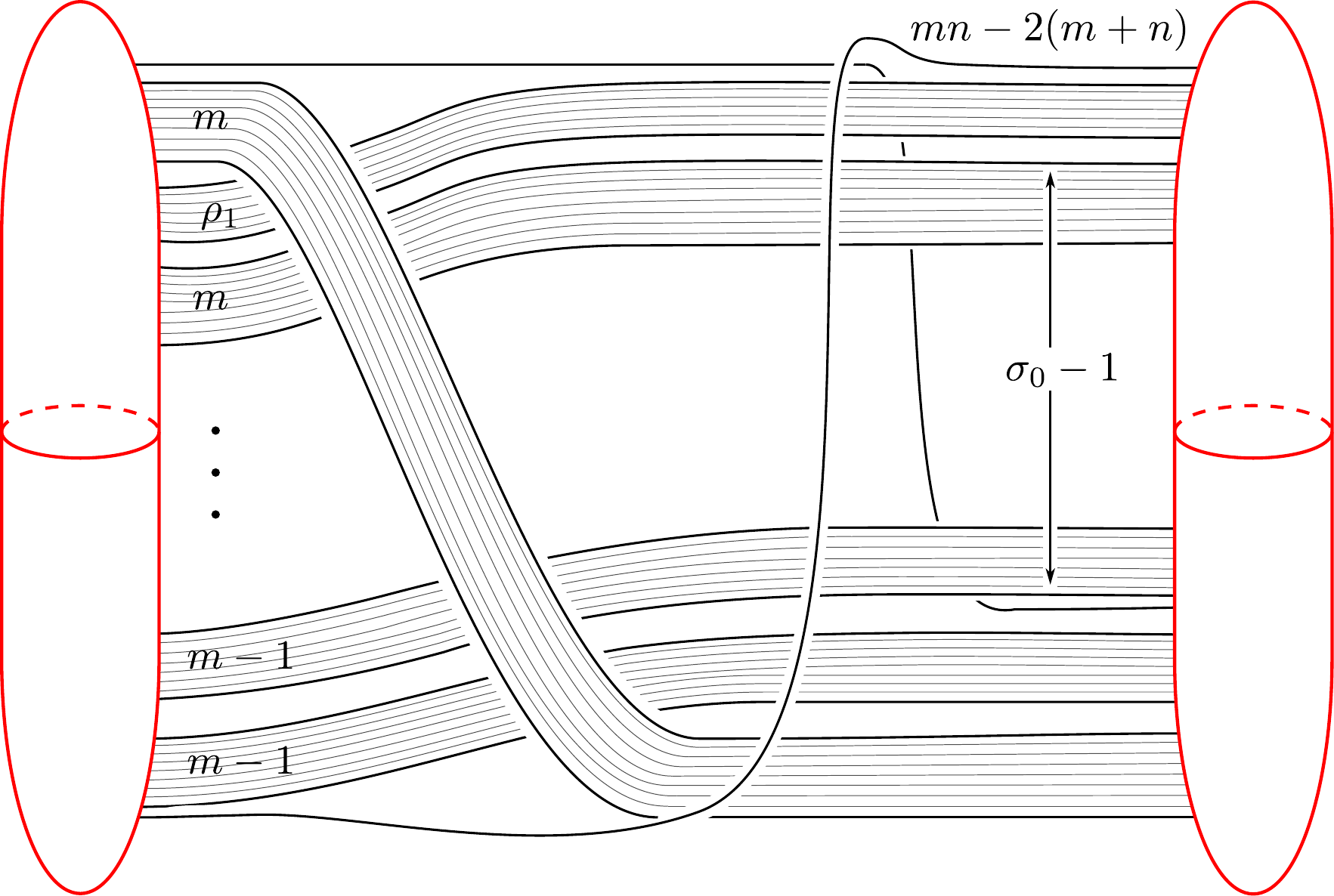}};
    		\end{tikzpicture}
			\caption{\small Sliding the 2-handle of $A_{m,n}$ under the 1-handle; 
						dragging the attaching circle of that
						2-handle around the 1-handle once.}
					\label{Figure: Amn Stein 1}
		\end{figure}	
		We refer to the portion of
		$K$ passing behind the central plane of the two 
		attaching balls of the 1-handle as the ``bad'' strand.
		We now pair off negative crossings in the bad strand 
		with positive crossings in $K$ by ``unraveling'' the 2-handle.  
		To accomplish this, begin by 
		dragging the bad strand over the 1-handle 
		(right side of Figure \ref{Figure: Amn Stein 1}).
		Each time we drag the bad strand over the 1-handle, we 
		unwind a strand off of the lowest remaining band of $m$ strands and
		wind that strand into a parallel band at the top
		 - thereby eliminating $m-1$ negative crossings
		at the expense of $m-1$ positive crossings.  
		Repeating $\sigma_0-1$ more times gives the left side
		of Figure \ref{Figure: Amn Stein 2}.
		\begin{figure}[!ht]
			\begin{tikzpicture}[xscale=.9]
				\node[inner sep=0pt] at (0,0)
    				{\includegraphics[scale=.42]{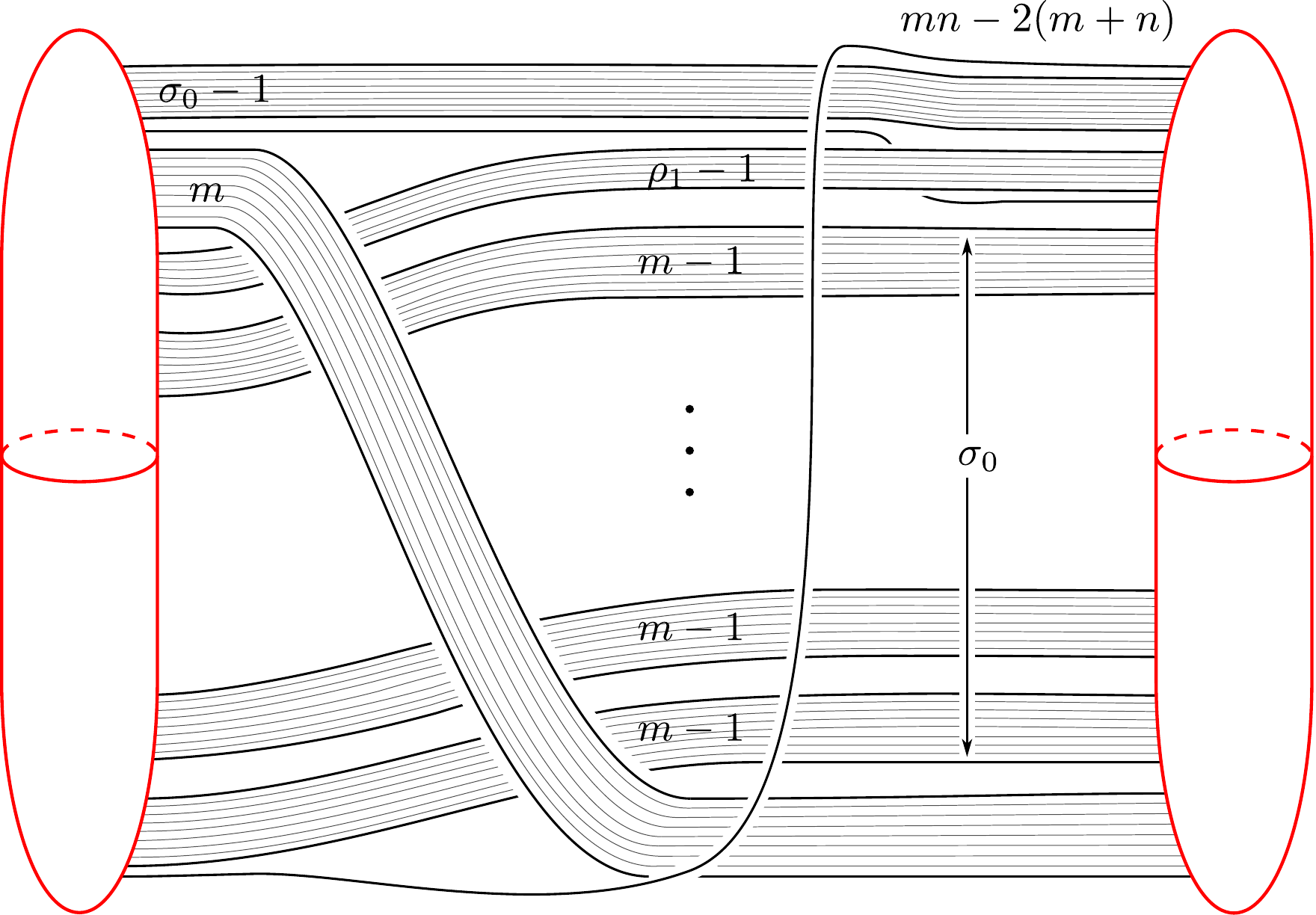}};
    			\node[inner sep=0pt] at (9,0)
					{\includegraphics[scale=.42]{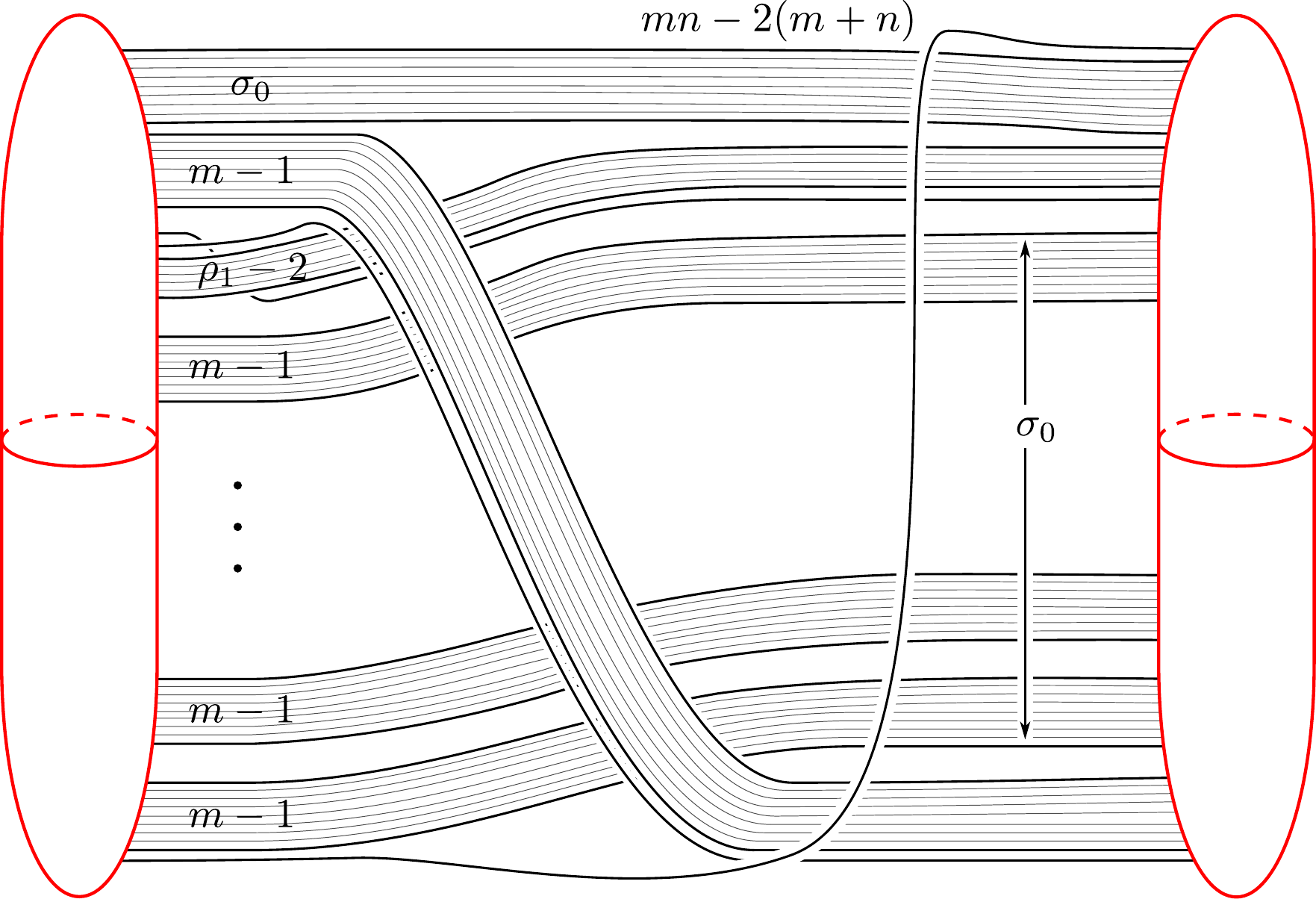}};
    		\end{tikzpicture}
			\caption{\small The result of dragging the attaching circle $\s_0$-times;
						and after $2\sigma_0+2$-times.}
			\label{Figure: Amn Stein 2}
		\end{figure}	
		We again push what remains of the bad strand around the 1-handle - 
		this time, a total of $\sigma_0+2$-times - giving the right side of Figure 
		\ref{Figure: Amn Stein 2}.
		We repeat the process of dragging the negative twist around the 
		1-handle $\sigma_0+2$ times.  Each time, the twist involves one less strand.
		After $k$ such iterations, the braid in the upper right of Figure 
		\ref{Figure: Amn Stein 2} is replaced by that of Figure 
		\ref{Figure: Amn Inductive Braid}.  
		\begin{figure}[!ht]
				\includegraphics[scale=.5]{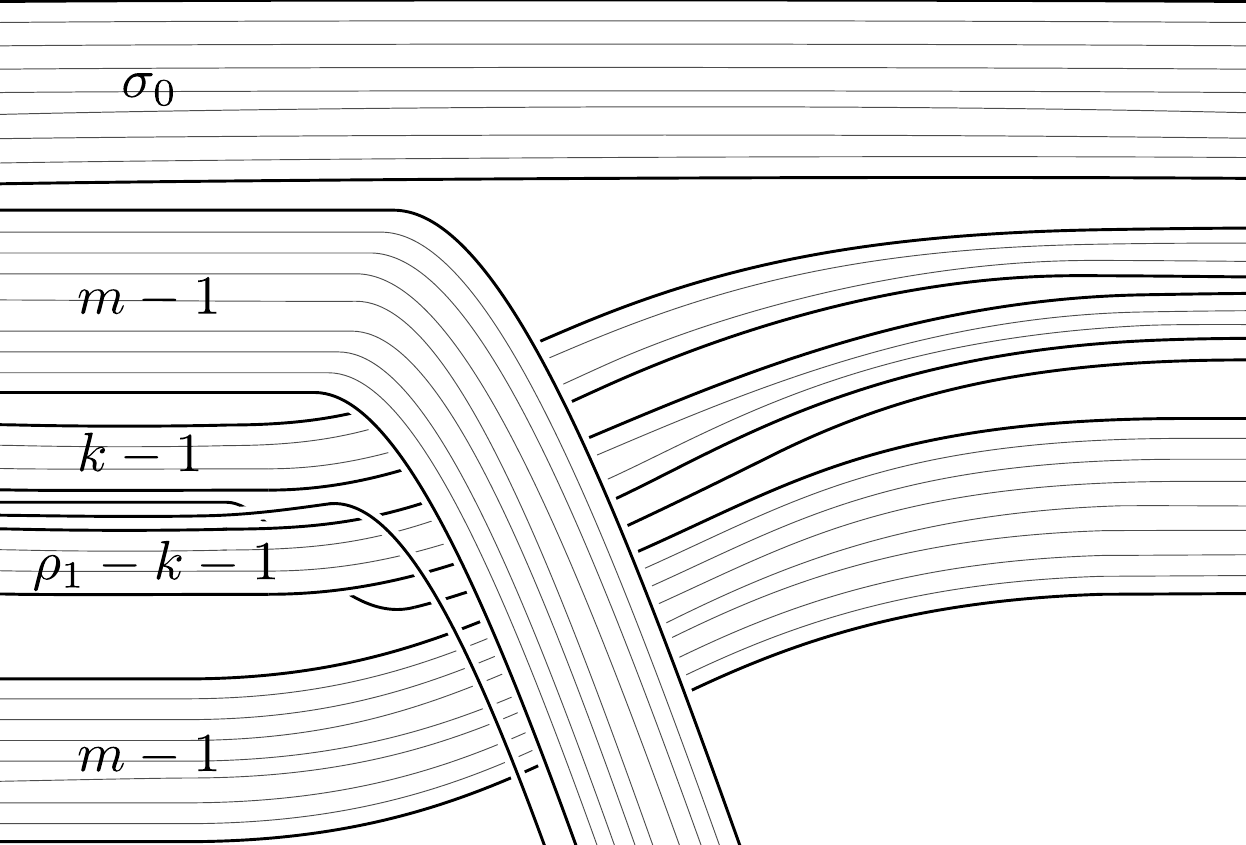}
					\caption{\small The result of dragging the 2-handle of $A_{m,n}$
					around the 1-handle $(k+1)\s_0+2$-times.}
					\label{Figure: Amn Inductive Braid}	
		\end{figure}	
		It's then immediate that $\rho_1-1$ iterations 
		gives Figure \ref{Figure: (A_m,n,J_1)}.
	\end{proof}
	\noindent				
	The fact that $(\partial A_{m,n},\xi_{J_1})$
	is contactomorphic to $(\partial B_{p,q},\xi_{J_{p,q}})$ and thus
	to $(L(p^2,pq-1),\bar{\xi}_{\mbox{st}})$ is
	Corollary \ref{xi_J is contactomorphic to xi_J_1}. 
	Also, it is worth noting that
	$\widetilde{J}_{1,p-1}$ is $J_{p-1,1}$.	
}
\section{Boundary Diffeomorphisms}\label{Section: Boundary Diffeomorphisms}
	{
	In this section, we exhibit explicit diffeomorphisms from $\partial B_{p,q}$
	and $\partial A_{m,n}$ to $L(p^2,pq-1)$.  To accomplish this, we find boundary
	diffeomorphisms to particular linear plumbings associated to $p$ and $q$
	(respectively $m$ and $n$).  These diffeomorphisms are needed to compare
	the resulting homotopy invariants of the contact structures induced by the Stein 
	structures on $B_{p,q}$ of \cite{Lekili-Bpq} 
	as well as those on $A_{m,n}$ coming from Proposition 
	\ref{Proposition: A Stein structure on A_m,n}.
	Along the way, we trace the meridian of the attaching circle of the single 
	2-handle of $B_{p,q}$ - proving 
	Theorem \ref{Theorem: Boundary Diffeomorphism}.
	
	It's worth noting that such diffeomorphisms have been
	known previously.  In \cite{Yamada-Amn}, Yamada produces 
	diffeomorphisms from $\partial A_{m,n}$ to  $L(p^2,pq-1)$
	expressed as the boundary of the unique linear plumbing of
	$D^2$-bundles over $S^2$ with Euler classes each $\leq -2$.  
	To accomplish this, one must carefully keep track of 
	every stage of the Euclidean algorithm applied to $(p-q,q)=1$
	- that is every time $a_i$ is subtracted from $b_i$ or
	$b_i$ from $a_i$ in Yamada's definition of $A(p-q,q)$
	(see Lemma \ref{Lemma: Defining A(p-q,q)}).  We perform a 
	courser bookkeeping of the Euclidean algorithm via Definition
	\ref{Definition: Euclidean Sequences}, which allows for
	an arguably clearer definition - however, we don't arrive at 
	a linear plumbing with Euler classes $\leq -2$.  Yet,
	through a sequence of blow-ups and cancellations, one can 
	easily get to that plumbing if so desired.  
	Furthermore, this definition lends itself to 
	defining the diffeomorphism from $\partial B_{p,q}$
	to $L(p^2,pq-1)$ as well:
	
	\begin{proposition}\label{Proposition: B_p,q Boundary Induction}
		Let $\{r_i\}_{i=-1}^{\ell+2}$ and $\{s_i\}_{i=0}^{\ell+1}$ 
		be as defined in Definition \ref{Definition: Euclidean Sequences}.
		Then for each $i\in\{0,\ldots,\ell+1\}$, 
		$B_{p,q} \stackrel{\partial}{\approx} B_{p,q}^i$
		where $B_{p,q}^i$ is the 4-manifold
		given by Figure \ref{Figure: B_p,q Boundary Induction)}.
		\begin{figure}[!ht]
			\centering
				\includegraphics[scale=.55]{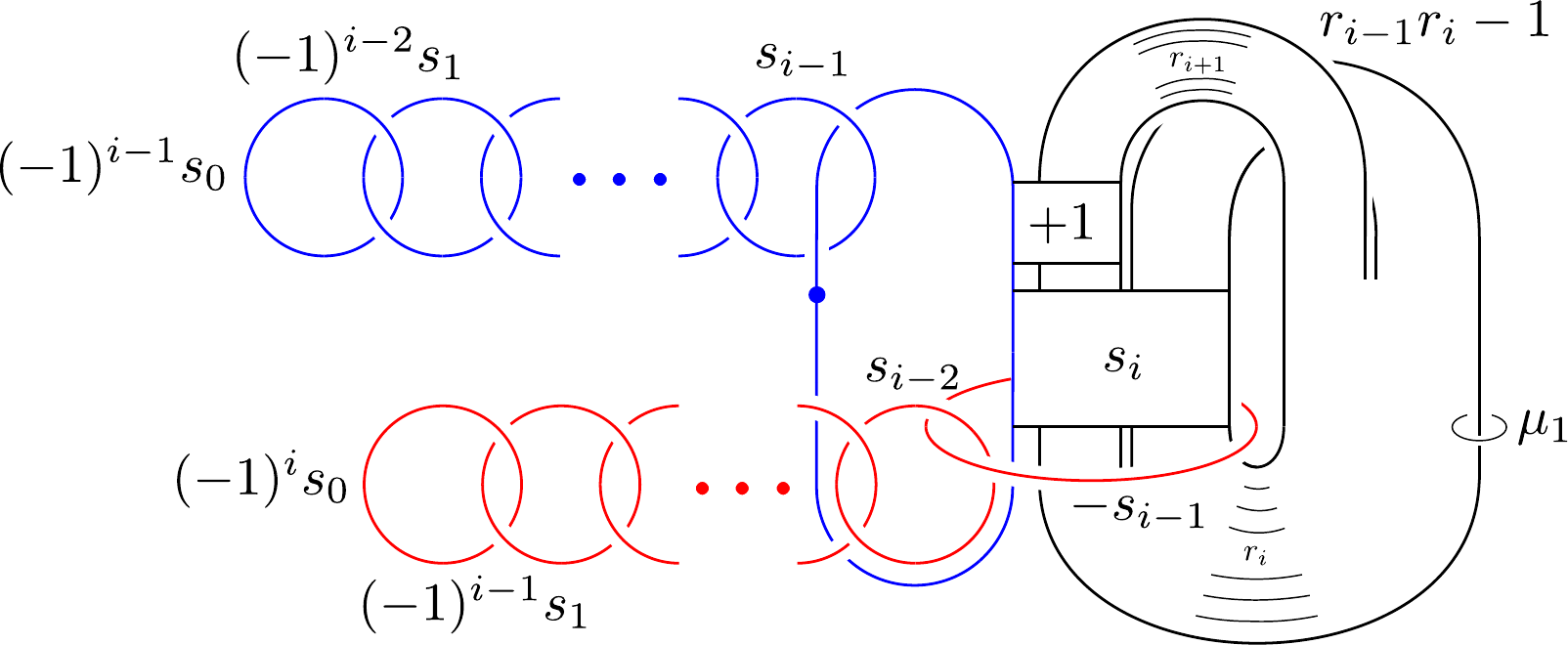}
			\caption{The 4-manifold $B_{p,q}^i$}\label{Figure: B_p,q Boundary Induction)}
		\end{figure}
	\end{proposition}
	\begin{proof}
		We induct on $i$.  
		When $i=0$, the result is immediate since $B_{p,q}^0 \approx B_{p,q}$. 
		Therefore, the proposition holds provided that 
		$\partial B_{p,q}^i \approx \partial B_{p,q}^{i+1}$.
		Let $K_1^i$ be the attaching circle of the $r_{i-1}r_i-1$-framed 
		2-handle in $B_{p,q}^i$.
		Suppose the result holds for some $i\leq\ell$.  For $i+1$,
		first, surger the single 1-handle and
		introduce a canceling pair of 1- and 2-handles to remove the 
		$s_i$-full twists between $K_1^i$ and the, now surgered, 1-handle
		(Figure \ref{Figure: B_p,q Boundary Induction Canceling Pair}).  
		\begin{figure}[!ht]
			\centering
				\includegraphics[scale=.55]{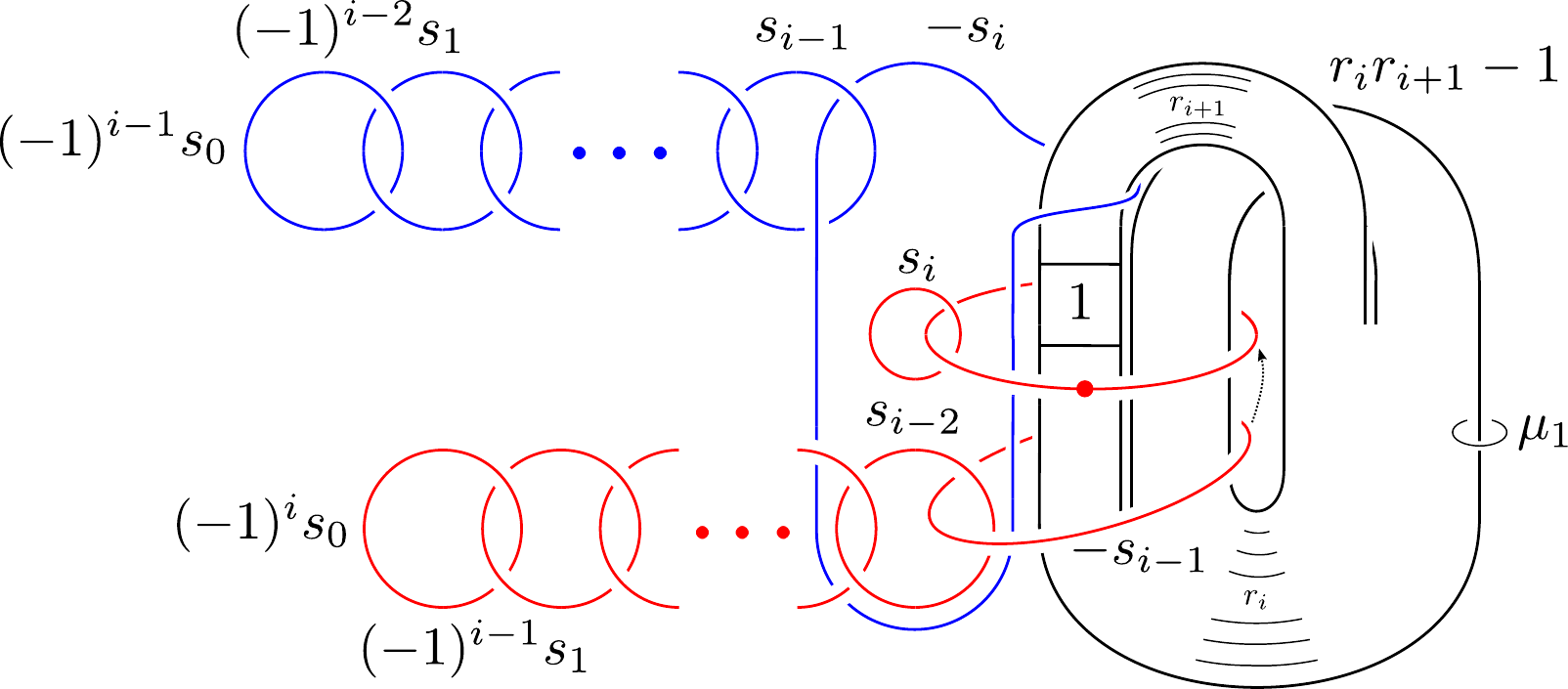}
			\caption{\small Introducing a canceling pair after surgery.}
				\label{Figure: B_p,q Boundary Induction Canceling Pair}
		\end{figure}
		Since $K_1^i$ links the new 1-handle $r_i$ times,
		the framing on $K_1^i$ decreases by $s_ir_i^2$ and
		the new framing on $K_1^i$ is
		\[
			r_{i-1}r_i-1 - s_ir_i^2 = r_i(r_{i-1}-s_ir_i) -1 = r_ir_{i+1}-1.
		\]				
		Sliding the $-s_{i-1}$-framed 2-handle under the new 1-handle as
		indicated in Figure \ref{Figure: B_p,q Boundary Induction Canceling Pair},
		and isotoping the $r_{i+1}$-stranded band (see Figure 
		\ref{Figure: B_p,q Boundary Induction Post Slide/Isotopy})
		\begin{figure}[!ht]
			\centering
				\includegraphics[scale=.55]{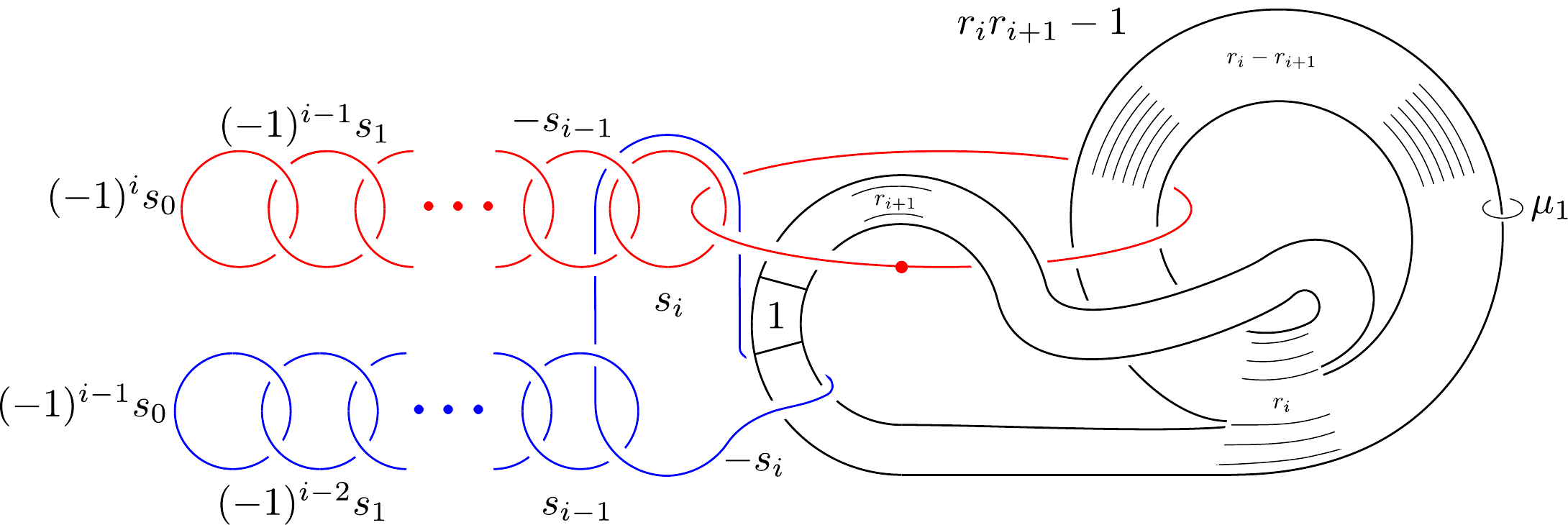}
			\caption{\small Isotoping $K_1^i$.}
				\label{Figure: B_p,q Boundary Induction Post Slide/Isotopy}
		\end{figure}
		we find that 
		the $r_{i+1}$-stranded band traverses the 1-handle (positively) 
		$s_{i+1}$-times as a complete band, while $r_{i+2}$-strands traverse 
		an additional one time to make up the complete 
		$s_{i+1}r_{i+1}+r_{i+2}=r_i$ linking.  With this view in mind, we isotope 
		$K_1^i$ into a closed braid on $r_{i+1}$ strands appropriately linking
		the carving disk of the 1-handle - Figure 
		\ref{Figure: B_p,q Boundary Induction Conclusion}.  The
		result holds by induction.
		\begin{figure}[!ht]
			\centering
				\includegraphics[scale=.55]{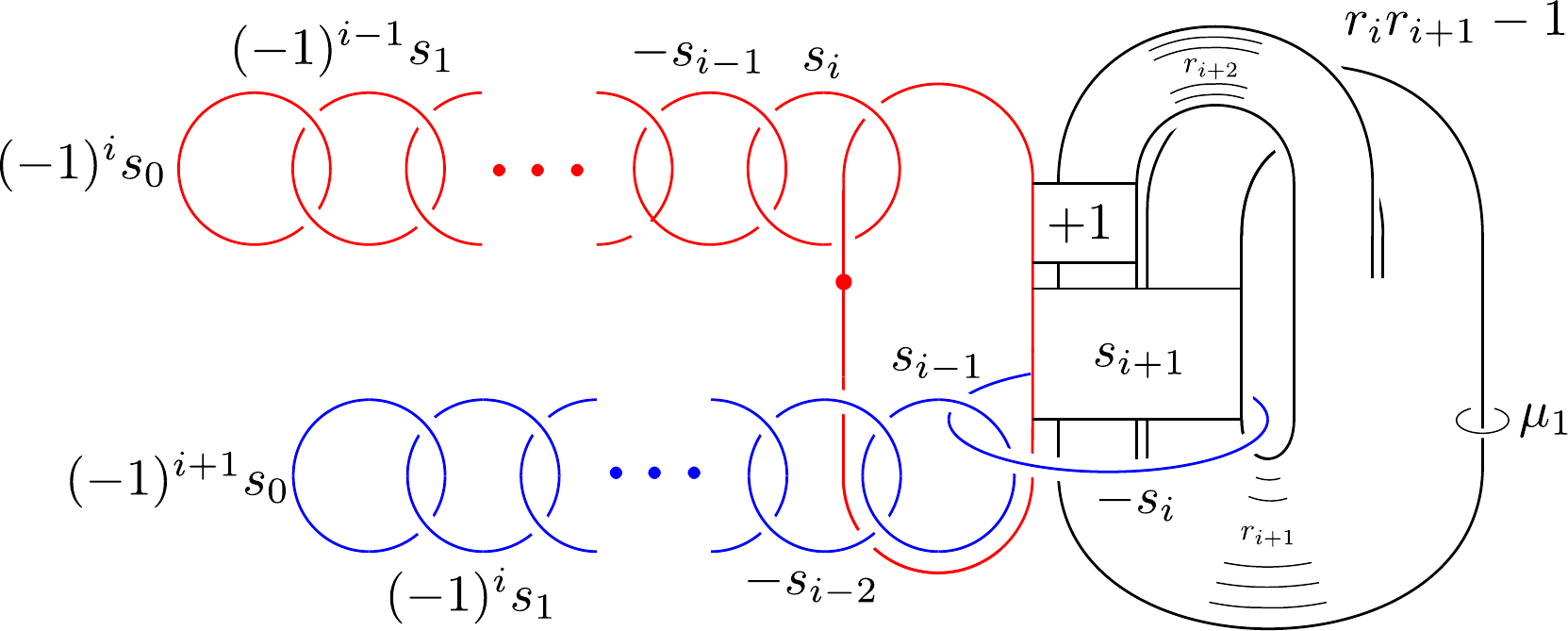}
			\caption{\small Further isotopy of $K_1^i$ to $K_1^{i+1}$}
				\label{Figure: B_p,q Boundary Induction Conclusion}
		\end{figure}		
	\end{proof}
	
	\begin{remark}
		At no point does $\mu_1$, the meridian of $K_1^i$, get damaged under the
		boundary diffeomorphisms defined in Proposition 
		\ref{Proposition: B_p,q Boundary Induction}.
		In particular, for each $i$, $\mu_1$ bounds a disk in $B_{p,q}^i$
		and the image of a collar neighborhood of $\mu_1$ arising from 
		such a disk persists under the boundary 
		diffeomorphisms defined above - that is that each diffeomorphism
		preserves the 0-framing on $\mu_1$.  
		
		Since $r_{\ell+1}=1$ and $r_{\ell+2}=0$, 
		by definition, $s_{\ell+1} = s_{\ell+1}r_{\ell+1}+r_{\ell+2} = r_\ell$.
		So, by looking at $B_{p,q}^{\ell+1}$ we arrive at the following
		result of Casson and Harer \cite{CassonHarer-RationalBalls}.
	\end{remark}
	\begin{cor}\label{Corollary: boundary of B_p,q is a lens space}
		$\partial B_{p,q}\approx L(p^2,pq-1)$.
	\end{cor}			
	\begin{proof}
		By Proposition \ref{Proposition: B_p,q Boundary Induction}, we have that 
		$\partial B_{p,q} \approx \partial B_{p,q}^{\ell+1}$ (Figure
		\ref{Figure: B_p,q Boundary Induction - Final Case)}).		
		\begin{figure}[!ht]
			\centering
				\includegraphics[scale=.55]{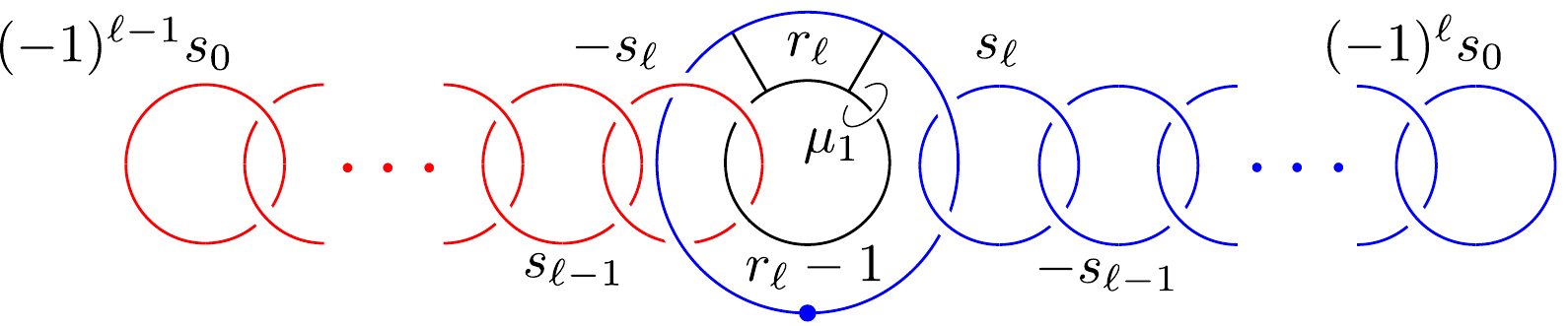}
			\caption{\small The space $B_{p,q}^{\ell+1}$.}
				\label{Figure: B_p,q Boundary Induction - Final Case)}
		\end{figure}
		We show that $\partial B_{p,q}^{\ell+1}$ is diffeomorphic to a linear plumbing
		of disk-bundles over $S^2$ as follows.  Surger the 1-handle and 
		introduce a canceling 1- and 2-handle, as
		in the induction step of Proposition 
		\ref{Proposition: B_p,q Boundary Induction},
		(top of Figure \ref{Figure: B_p,q Boundary Induction - Final Case2)}).
		Next, slide the
		$-s_\ell$-framed 2-handle as well as $\mu_1$ under the 1-handle as indicated
		in the top of Figure \ref{Figure: B_p,q Boundary Induction - Final Case2)}
		(middle of Figure \ref{Figure: B_p,q Boundary Induction - Final Case2)}).  
		Surgering the new 1-handle and blowing down gives the linear plumbing 
		(bottom of Figure \ref{Figure: B_p,q Boundary Induction - Final Case2)}).  
		\begin{figure}[!ht]
			\begin{tikzpicture}[xscale=.9]
				\node[inner sep=0pt] at (0,2.25)
    				{\includegraphics[scale=.55]{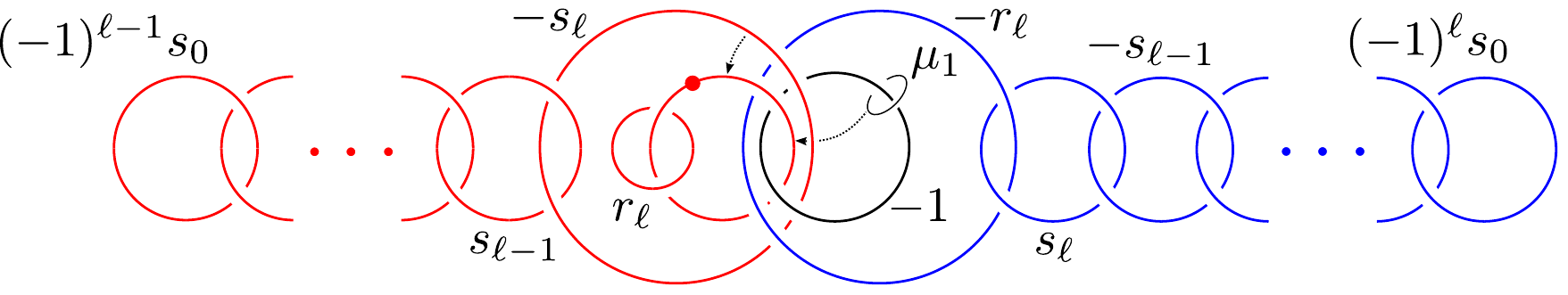}};
    			\node[inner sep=0pt] at (0,0)
					{\includegraphics[scale=.55]{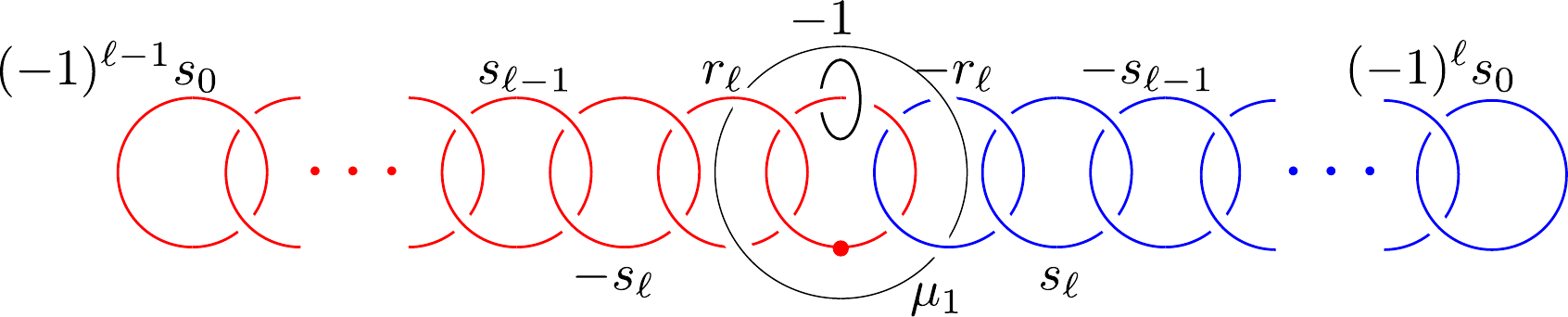}};
				\node[inner sep=0pt] at (0,-2.25)
		    		{\includegraphics[scale=.55]{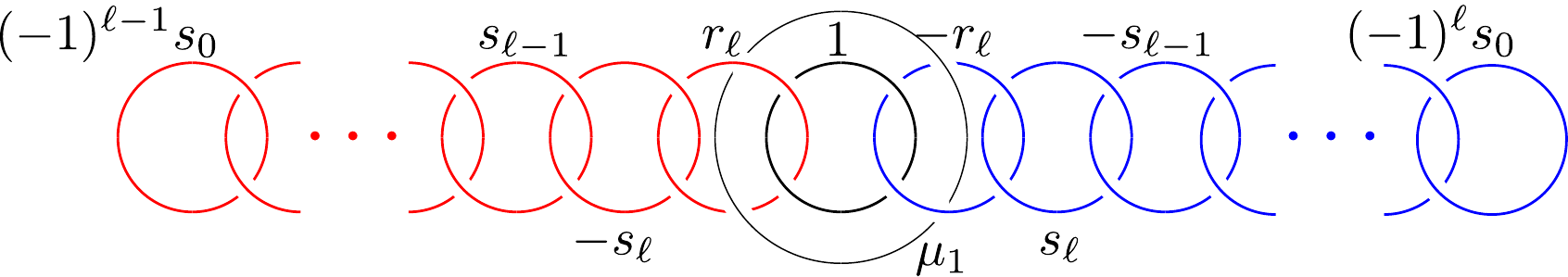}};
    		\end{tikzpicture}
			\caption{\small From top to bottom: The introduction of a canceling pair to
						$B_{p,q}^{\ell+1}$ after surgery; the result of
						the indicated slides; a linear plumbing
						associated to $\partial B_{p,q}$.}
				\label{Figure: B_p,q Boundary Induction - Final Case2)}
		\end{figure}
	\end{proof}
	\begin{remark}\label{Remark: B_p,q 0-framing preserved}
		From Lemma \ref{Lemma: intersection form algebra}, we
		see that the above linear plumbing bounds $L(p^2,pq-1)$.  Indeed
		\begin{align*}
			\left[-s_0,s_1,\ldots,  \pm r_\ell ,1 , 
				\mp r_\ell,\ldots, -s_1,s_0\right] 
					= -\frac{p^2}{pq-1}.
		\end{align*}
		Notice also that the image of $\mu_1$ is given as the 0-framed 
		push-off of the attaching circle of the central 1-framed unknot. 
		We'll trace where the curve, $\gamma$ in Figure \ref{Figure: Boundary Diff},
		goes as well - finding that it too goes to the 0-framed push-off
		of the central 1-framed unknot via an appropriately
		defined diffeomorphism.  To define this diffeomorphism, in a structurally 
		similar manner to that of Proposition \ref{Proposition: B_p,q Boundary Induction},
		we note the following fact about $A_{m,n}$.		
	\end{remark}
	
	\begin{lem}\label{Lemma: Shapes of Amn}
		$A_{m,n}$ is given by Figure \ref{Figure: Alernative A_m,n}.
		\begin{figure}[!ht]
			\centering
				\includegraphics[scale=.55]{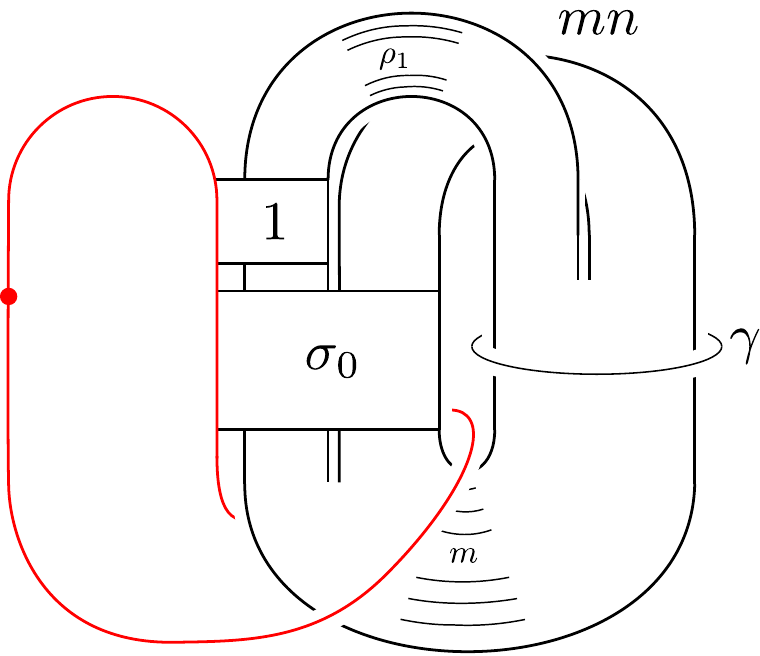}
			\caption{\small An alternative description of $A_{m,n}$.}
				\label{Figure: Alernative A_m,n}
		\end{figure}			
	\end{lem}
	\begin{proof}
		As before, we are taking $n=m\sigma_0+\rho_1$. 
		The result follows from an isotopy of the 2-handle.
	\end{proof}
	
	\begin{proposition}\label{Proposition: A_m,n Boundary Induction}
		Let $\{\rho_i\}_{i=-1}^{\ell+2}$ and $\{\sigma_i\}_{i=0}^{\ell+1}$ 
		be as defined in Definition \ref{Definition: Euclidean Sequences} 
		(associated to $n>m\geq1$).  Then for each $i\in\{0,\ldots,\ell+1\}$, 
		 $A_{m,n} \stackrel{\partial}{\approx} A_{m,n}^i$
		where $A_{m,n}^i$ is the 4-manifold
		given by Figure \ref{Figure: A_m,n Boundary Induction}.
		\begin{figure}[!ht]
			\centering
				\includegraphics[scale=.55]{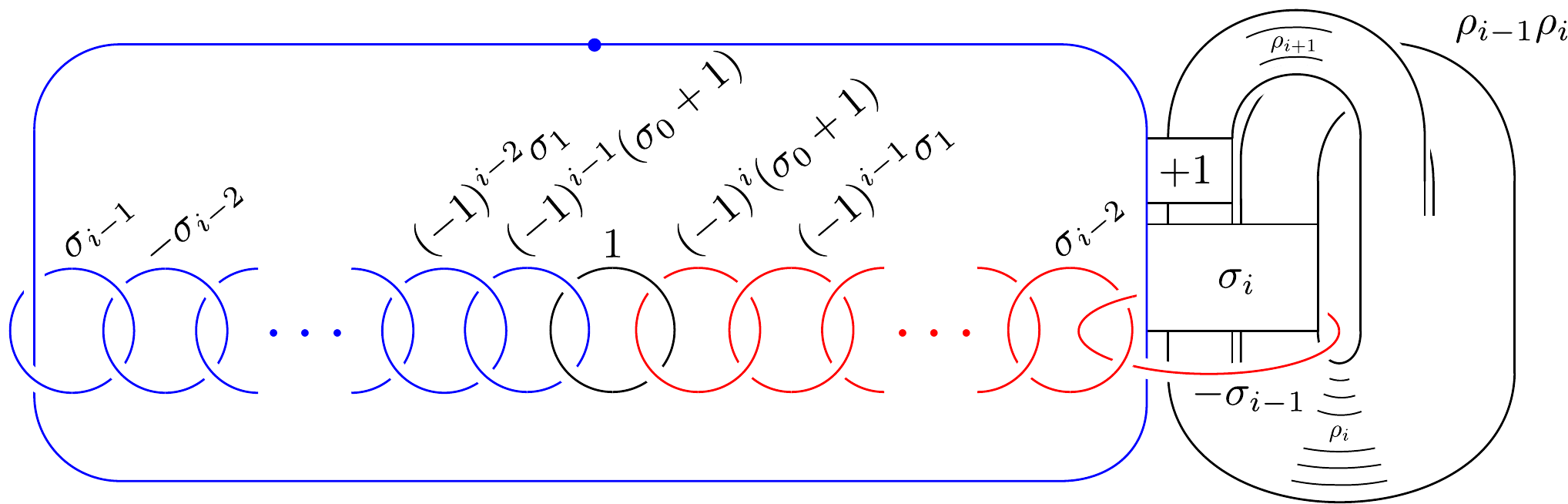}
			\caption{\small The 4-manifold $A_{m,n}^i$}
			\label{Figure: A_m,n Boundary Induction}
		\end{figure}
	\end{proposition}
	
	\begin{proof}
		We induct on $i$, treating the base case and the induction step
		simultaneously.  For the base case, start with the handle decomposition
		from Lemma \ref{Lemma: Shapes of Amn}.  For the induction step,
		suppose that the result holds for some $i\leq \ell$.  Let $K_1^i$ be the 
		attaching circle of the $\rho_{i-1}\rho_{i}$-framed 2-handle in $A_{m,n}^i$.
		Surger the 1-handle and introduce a canceling 1- and 2-handle 
		(for the base case see the left side of Figure 
		\ref{Figure: A_m,n Boundary Induction) - Base Case},
		for the induction step see Figure \ref{Figure: A_m,n Boundary Induction - slide}).  
		Notice, similar to Proposition \ref{Proposition: B_p,q Boundary Induction}
		the framing of $K_1^i$ changes from $\rho_{i-1}\rho_i$ 
		to $\rho_i\rho_{i+1}$. 
		\begin{figure}[!ht]
			\begin{tikzpicture}[xscale=.9]
				\node[inner sep=0pt] at (0,3)
    				{\includegraphics[scale=.55]{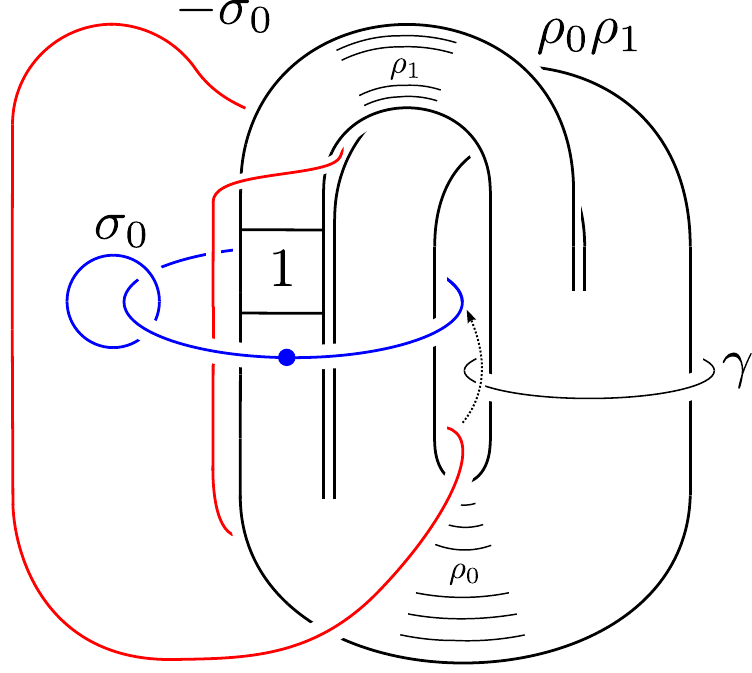}};
				\node[inner sep=0pt] at (5,3)
    				{\includegraphics[scale=.55]{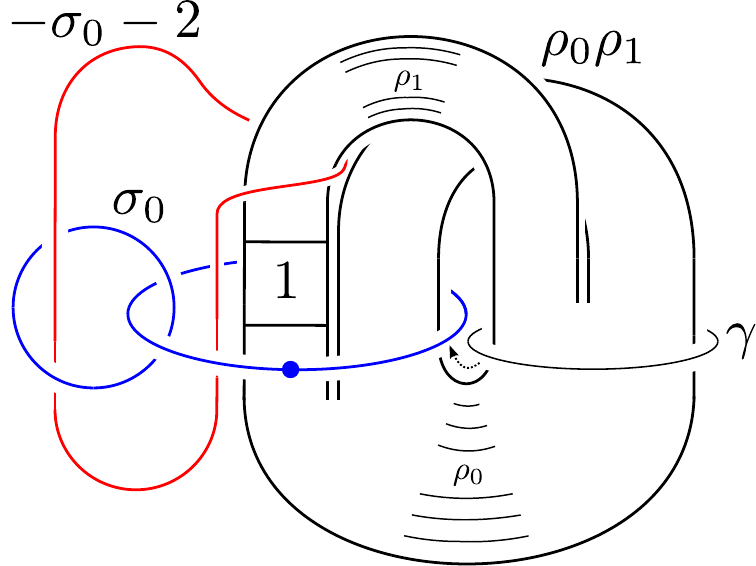}};
				\node[inner sep=0pt] at (10.5,3)
    				{\includegraphics[scale=.55]{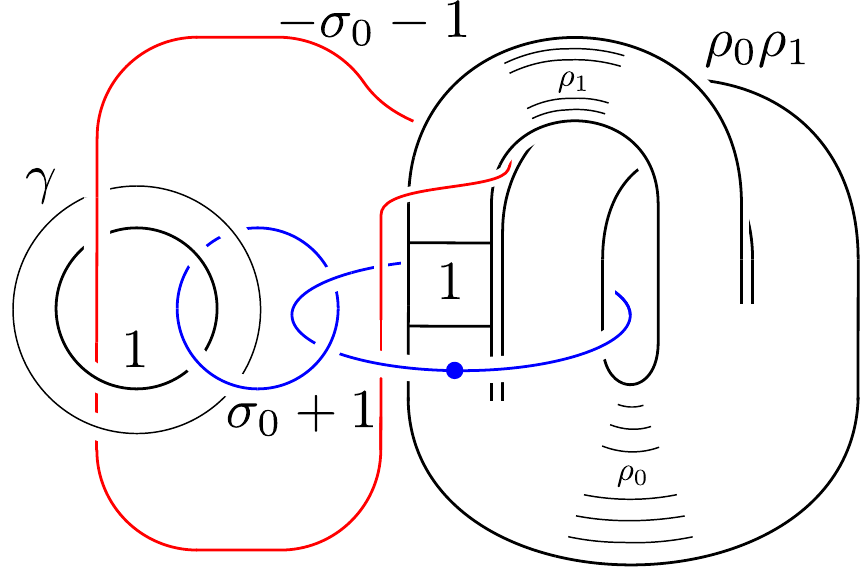}};
			\end{tikzpicture}\centering
			\caption{\small The base case of Proposition 
				\ref{Proposition: A_m,n Boundary Induction}}
					\label{Figure: A_m,n Boundary Induction) - Base Case}
		\end{figure}
		\begin{figure}[!ht]
			\centering
				\includegraphics[scale=.6]{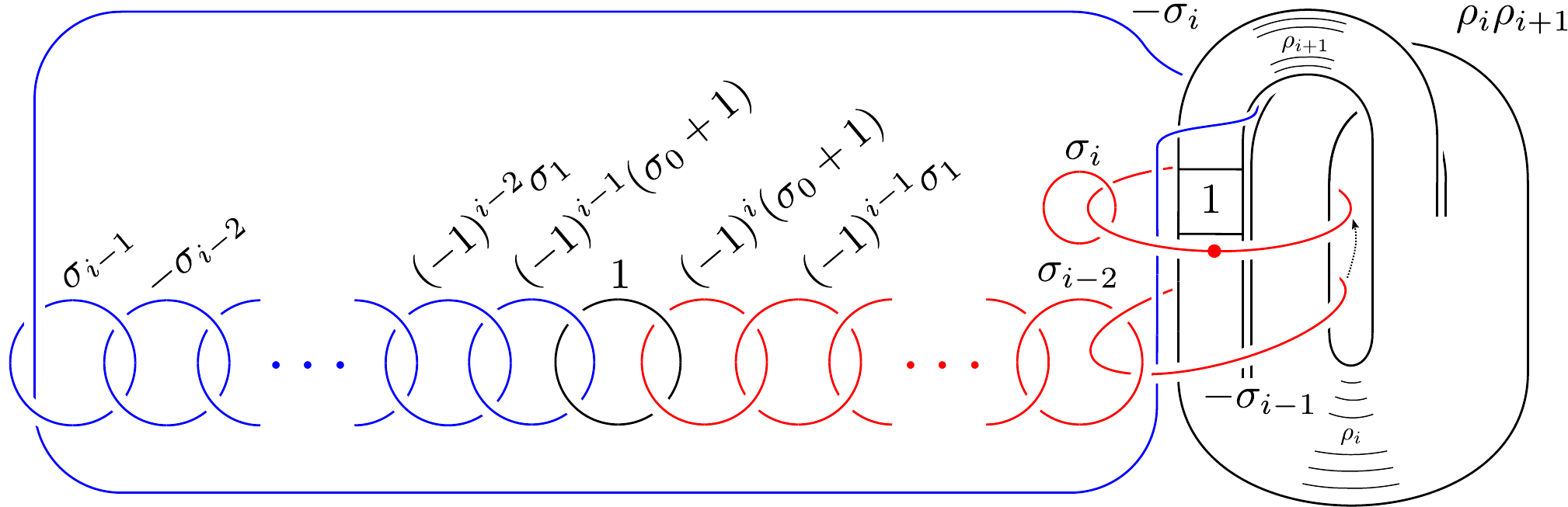}
			\caption{\small Introducing a canceling pair.}
			\label{Figure: A_m,n Boundary Induction - slide}
		\end{figure}
		Slide the now surgered 1-handle as indicated in the respective figures
		and, for the base case, blow-up once 
		(right side of Figure \ref{Figure: A_m,n Boundary Induction) - Base Case}).
		From here the base case follows similarly to the induction step;
		both of which are structurally similar to Proposition 
		\ref{Proposition: B_p,q Boundary Induction}.  Indeed, isotope $K_1^i$ 
		to view a band with $\rho_{i+1}$ stands traversing the 1-handle 
		$\sigma_{i+1}$-times along with $\rho_{i+2}$ of those strands traversing 
		an extra time as in Figure \ref{Figure: A_m,n Boundary Induction-Isotopy}.
		\begin{figure}[!ht]
			\centering
				\includegraphics[scale=.55]{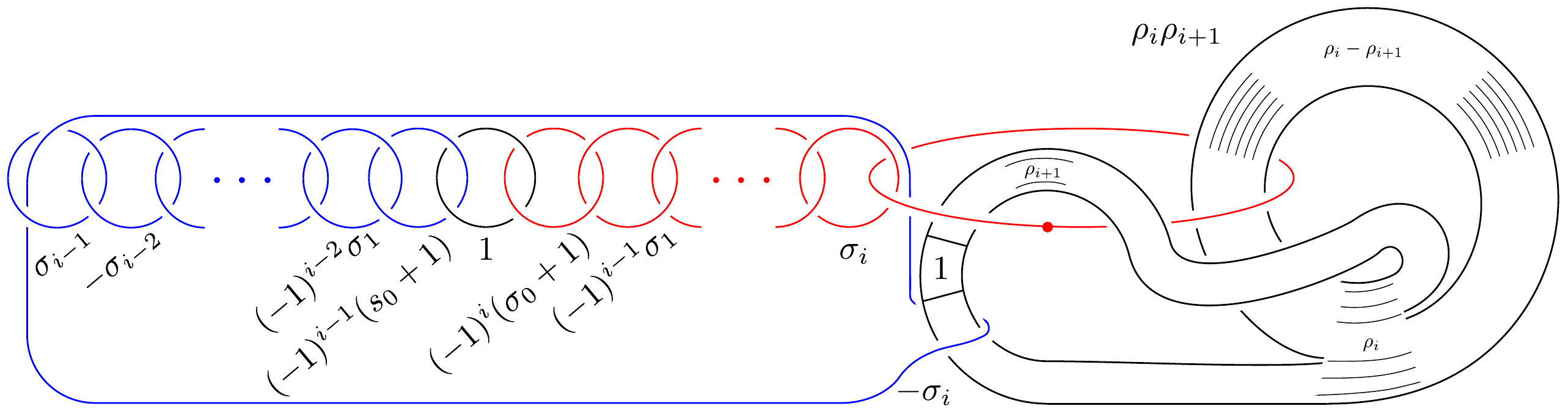}
			\caption{\small Isotoping $K_1^i$ in $A_{m,n}^i$.}
				\label{Figure: A_m,n Boundary Induction-Isotopy}
		\end{figure}
		A further isotopy of $K_1^i$ gives a closed braid on $\rho_{i+1}$-strands
		geometrically linking the carving disk of the new 
		1-handle $\rho_i$-times.  Finally, notice that 
		to get the appropriate linking on the chain of unknots,
		we have to wind the chain (as indicated in Figure
		\ref{Figure: A_m,n Boundary Induction - final}) to add a total
		of $i$ positive half-twists to the left of the euler-class 1 disk-bundle
		along with $i$ negative half-twists to the right.
		The result follows by induction.
		\begin{figure}[!ht]
			\centering
				\includegraphics[scale=.55]{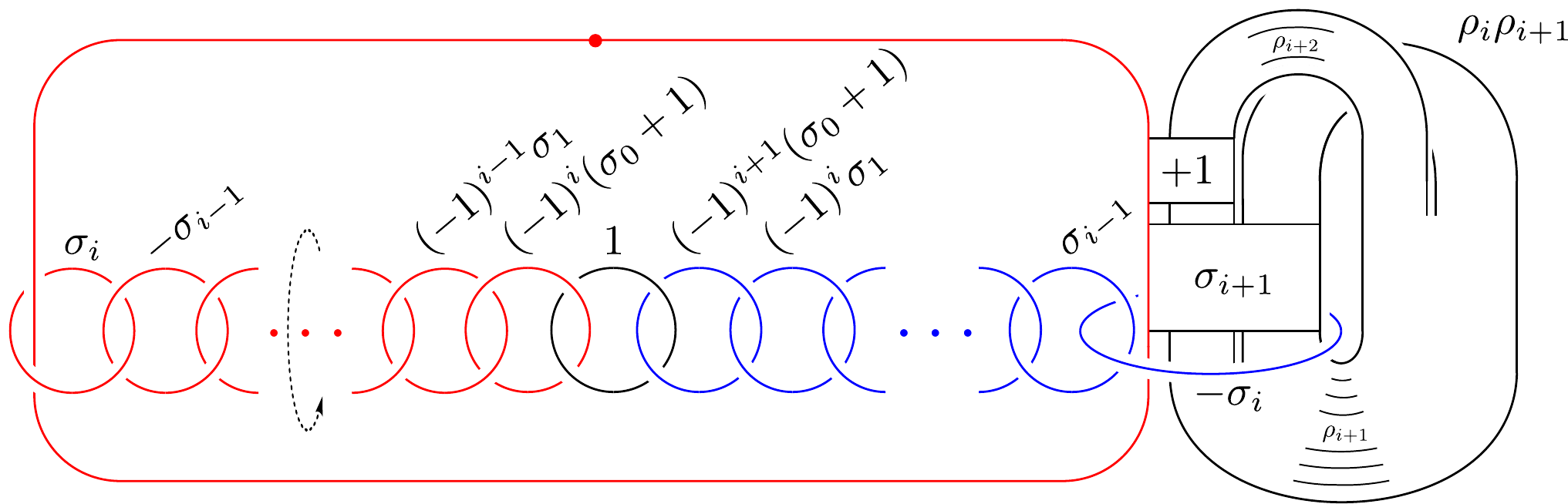}
			\caption{\small Further isotopy of $K_1^i$ to $K_1^{i+1}$ in $A_{m,n}^{i+1}$.}
			\label{Figure: A_m,n Boundary Induction - final}
		\end{figure}
	\end{proof}

	\begin{cor}[\cite{Yamada-Amn} Theroem 1.1]\label{Corollary: boundary of A_m,n is a lens space}
		$\partial A_{m,n}\approx L(p^2,pq-1)$ for $(p-q,q)=A(m,n)$.
	\end{cor}
	\begin{proof}
		By Proposition \ref{Proposition: A_m,n Boundary Induction}, 
		$\partial A_{m,n} \approx \partial A_{m,n}^{\ell+1}$ 
		(figure \ref{Figure: A_m,n^l+1}).
		\begin{figure}[!ht]
			\centering
				\includegraphics[scale=.55]{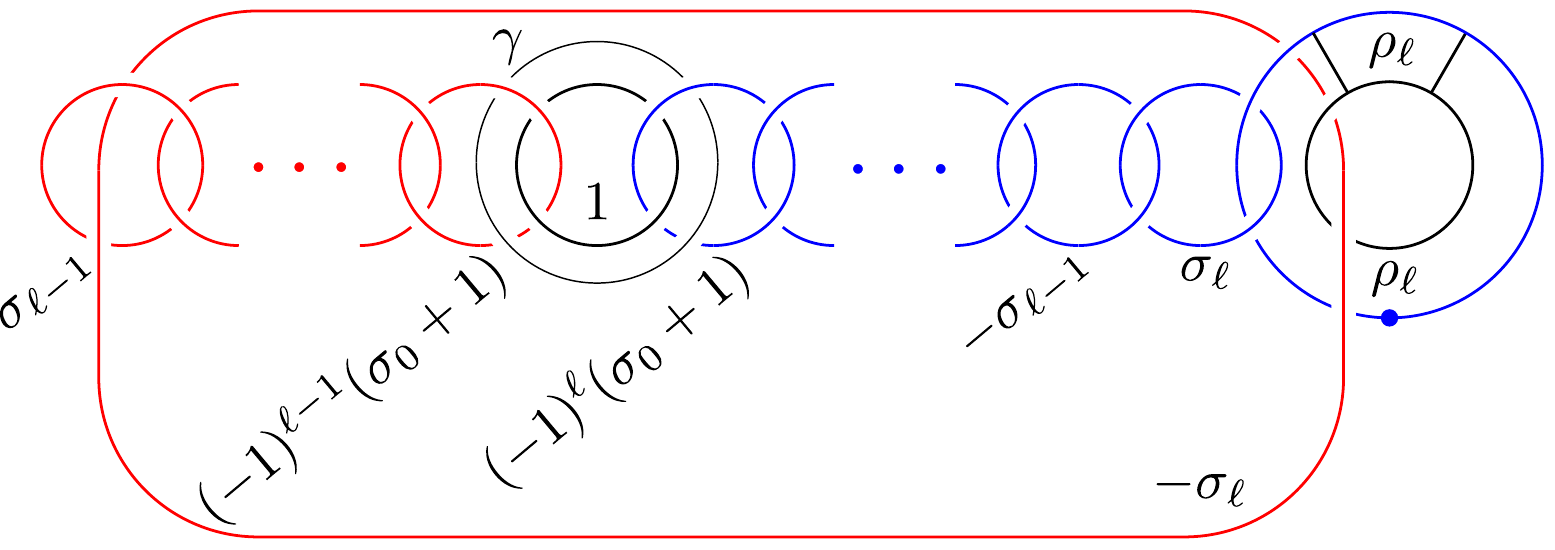}
			\caption{\small  The space $A_{m,n}^{\ell+1}$}\label{Figure: A_m,n^l+1}
		\end{figure}
		We proceed as in Corollary \ref{Corollary: boundary of B_p,q is a lens space}.
		\begin{figure}[!ht]
			\begin{tikzpicture}[xscale=.9]
				\node[inner sep=0pt] at (0,3.5)
    				{\includegraphics[scale=.55]
	    				{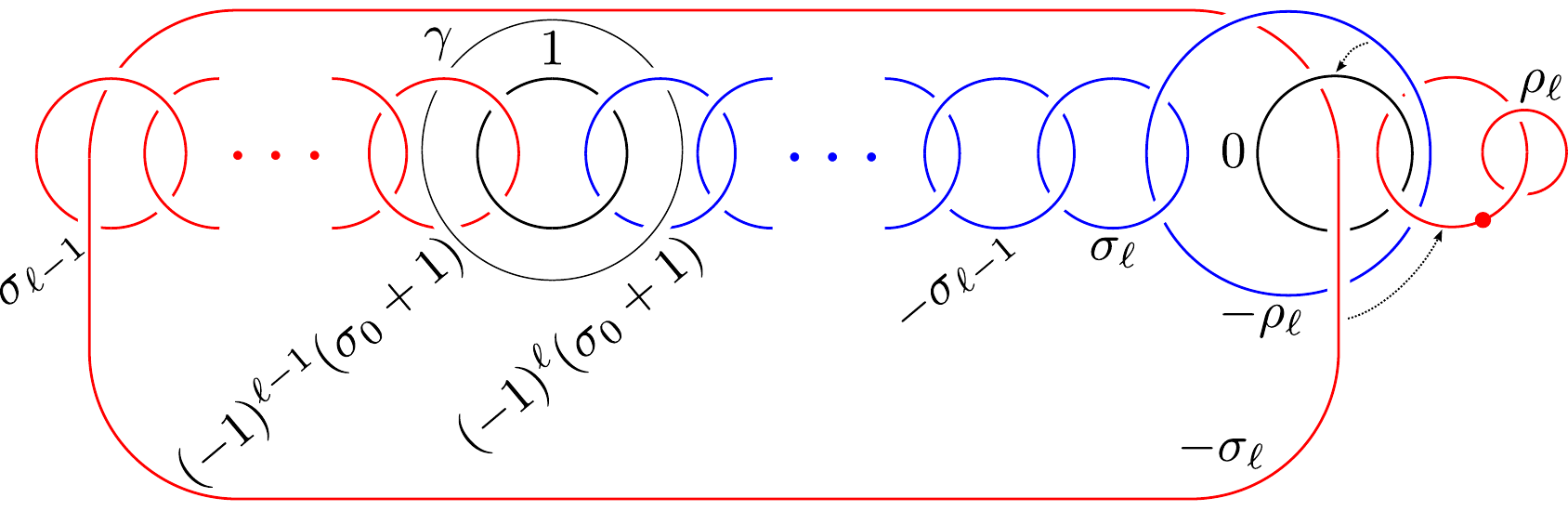}};
    			\node[inner sep=0pt] at (0,0)
					{\includegraphics[scale=.55]
						{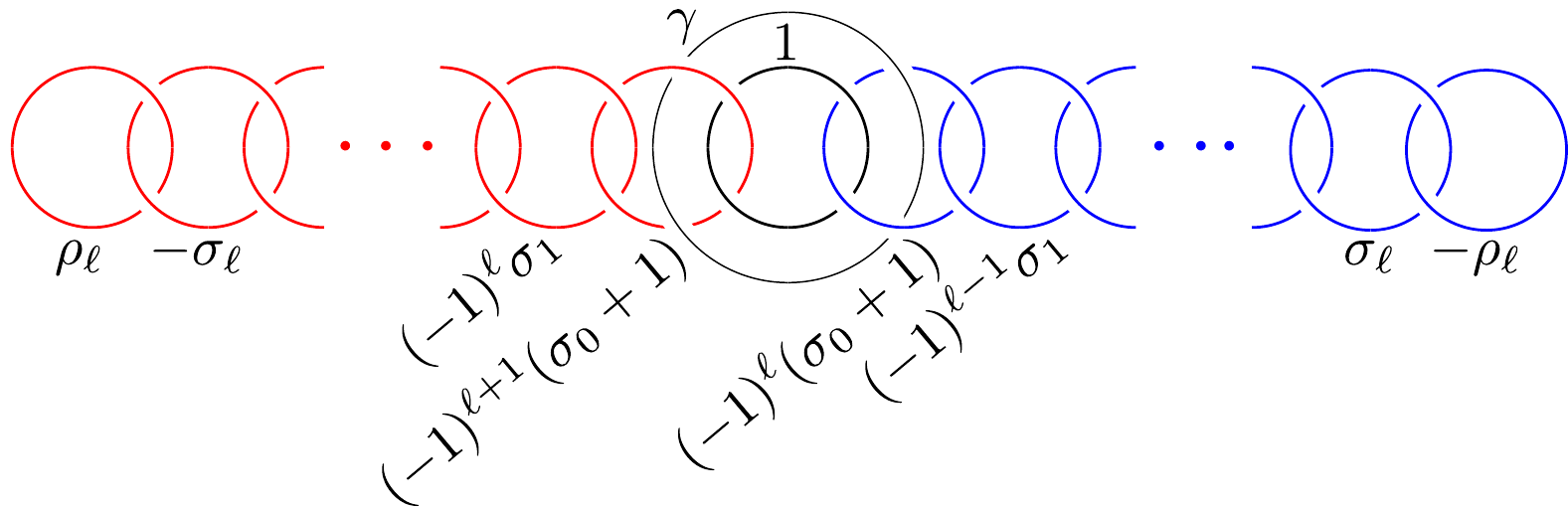}};				
    		\end{tikzpicture}
			\caption{\small The result of surgering $A_{m,n}^{\ell+1}$ and introducing
						a canceling pair; a linear plumbing associated to 
						$\partial A_{m,n}$}\label{Figure: A_m,n Boundary Induction Final Case}
		\end{figure}
		After surgering the 1-handle and 
		introducing a canceling 1- and 2-handle
		(top of Figure \ref{Figure: A_m,n Boundary Induction Final Case} ), 
		slide the $-\sigma_\ell$-framed 2-handle under the 1-handle and 
		the $-\rho_\ell$-framed 2-handle over the 0-framed 2-handle
		as indicated in the top of Figure \ref{Figure: A_m,n Boundary Induction Final Case}.
		Canceling the 1-handle with the 0-framed 2-handle gives
		the linear plumbing (bottom of Figure 
		\ref{Figure: A_m,n Boundary Induction Final Case}).
	\end{proof}		
	\begin{remark}\label{Remark: A_m,n 0-framing preserved}
		The fact that $\partial A_{m,n}$ is $L(p^2,pq-1)$ for $A(m,n)=(p-q,q)$
		follows by noting that given $p$ and $q$,
		or equivalently $m$ and $n$, we can define the other pair by an 
		appropriate identification of the linear plumbings in Corollaries
		\ref{Corollary: boundary of B_p,q is a lens space} and
		\ref{Corollary: boundary of A_m,n is a lens space} - provided
		that $s_0>1$ (that is that $p-q>q$).  In fact,
		this could be taken as the definition of the function $A$ defined 
		in \cite{Yamada-Amn}.  The latter claim is the content of Lemma 
		\ref{Lemma: Defining A(p-q,q)}.  
		Notice also that $\gamma$ bounds a disk in each 
		$\partial A_{m,n}^i$ as well as in the linear plumbing of 
		Figure \ref{Figure: A_m,n Boundary Induction Final Case}.
		Furthermore, each boundary diffeomorphism defined 
		in Proposition \ref{Proposition: A_m,n Boundary Induction}
		and those of Corollary \ref{Corollary: boundary of A_m,n is a lens space}
		preserve the 0-framing of $\gamma$ specified by those disks.
	\end{remark}
	
	\begin{proof}[Proof of Theorem \ref{Theorem: Boundary Diffeomorphism}]
		As $A(p-q,q)=(m,n)$, we can identify the plumbings of 
		Figures \ref{Figure: B_p,q Boundary Induction - Final Case2)}
		and \ref{Figure: A_m,n Boundary Induction Final Case}.		
		Then, by first, applying the diffeomorphisms of Proposition
		\ref{Proposition: B_p,q Boundary Induction} we get
		a diffeomorphism from $\partial B_{p,q}$ to the boundary
		of the linear plumbing of the bottom
		of Figure \ref{Figure: B_p,q Boundary Induction - Final Case2)}\
		caring $\mu_1$ as indicated.
		Then applying the diffeomorphisms of Proposition 
		\ref{Proposition: A_m,n Boundary Induction} in reverse 
		from the boundary of the linear plumbing of 
		Figure \ref{Figure: A_m,n Boundary Induction Final Case}
		to $A_{m,n}$ gives the required diffeomorphism 
		$f:\partial B_{p,q}\to \partial A_{m,n}$.  The fact that carving the
		disk bounding $f(\mu_1)$ gives $S^1\times B^3$ follows by
		repeatedly sliding the now two 1-handles past each other and 
		canceling one with the single 2-handle of $A_{m,n}$.
	\end{proof}	
	\subsection{Spin Structures and Orientations}
	{
		We determine how $f$ behaves with respect to 
		elements of $H_1(\partial B_{p,q})$ as well as how $f$
		treats spin structures. Both of these behaviors will be 
		important.  
		\begin{lem}\label{Lemma: Moving Generators}
			Suppose that $L(p,q)$ is given by the linear plumbing
			\[
				\includegraphics[scale=.66]{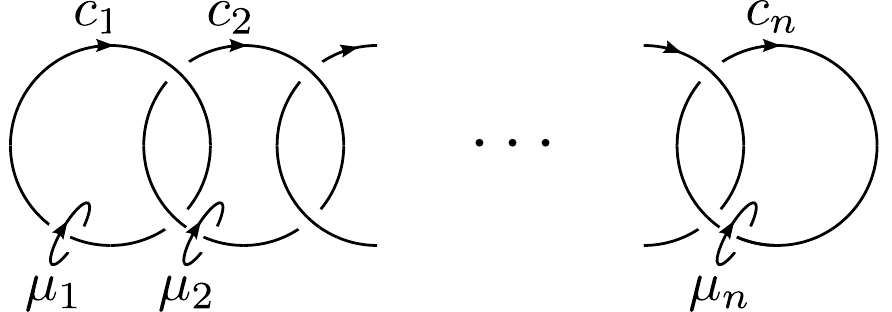}
			\]
			where the $\mu_i$'s are meridians spanning $H_1(L(p,q),\Z)$.  Then
			\[
				H_1(L(p,q),\Z) = \< \mu_1 : (\det C_n)\mu_1=0 \>
			\]
			where 
			$C_i \doteq 
				\begin{pmatrix}
					c_1	& 1  	& 		&   \\
					1	& c_2	& 1		&	\\
						& 1		& \ddots& 1	\\
					 	&		& 1		& c_i
				\end{pmatrix}$ and for $i\in\{2,\ldots,n\}$, 
				$\mu_i=(-1)^{i-1}\left(\det C_{i-1}\right)\mu_1$.
		\end{lem}
		\begin{proof}
			Given a Dehn surgery description of a 3-manifold, 
			one obtains a presentation for the first homology in terms 
			of the right handed meridians of the (oriented) framed link
			(see \cite{Gompf-and-Stipsicz} page 165).  In the above 
			case, we find that
			\[
				H_1(L(p,q),\Z) = \< \mu_1,\ldots,\mu_n : 
					\mu_2=-c_1\mu_1,
					\{\mu_{i+1}=-c_i\mu_i-\mu_{i-1}\}_{i=2}^{n-1}, 
					c_n\mu_n=-\mu_{n-1}\>
			\]		
			As $\mu_2=-c_1\mu_1 = (-1)^{2-1}(\det C_{2-1})\mu_1$, 
			the result follows by induction using that 
			\[
				\det C_k = c_k \det C_{k-1} - \det C_{k-2}.
			\]
		\end{proof}
		\begin{remark}\label{Remark: Homology Pullback}
		Lemma \ref{Lemma: Moving Generators} allows us to 
		determine $f^{-1}_*\gamma_0 \in H_1(\partial B_{p,q})$.  
		From Proposition 
		\ref{Proposition: A_m,n Boundary Induction}, we have that 
		a meridian of $-(\sigma_0+1)$-framed unknot
		of figure \ref{Figure: A_m,n Boundary Induction Final Case}
		is carried to $\gamma_0$ in $\partial A_{m,n}$.  
		Similarly, $\mu_0$ is carried to a meridian of
		$-s_0$-framed unknot of Figure 
		\ref{Figure: B_p,q Boundary Induction - Final Case2)}.
		Furthermore, by Corollary \ref{Corollary: pulling back gamma0},
		we have that $\gamma_0 = \pm n\mu_0$ if $\ell\in 2\Z$
		and $\gamma_0=\pm m\mu_0$ if $\ell\in 2\Z+1$
		where we view $\gamma_0$ and $\mu_0$ as their
		respective images in the aforementioned linear 
		plumbings.  Now, by an appropriate choice of 
		identification of the plumbings of Figures 
		\ref{Figure: A_m,n Boundary Induction Final Case} 
		and \ref{Figure: B_p,q Boundary Induction - Final Case2)}
		we can always assume that 
		\[
			f^{-1}_*\gamma_0 = \left\{
				\begin{array}{ll}
					+n\mu_0	&	\mbox{if $\ell\in 2\Z$,}\\
					+m\mu_0	&	\mbox{if $\ell\in 2\Z+1$}.
				\end{array}\right.
		\]
		Indeed, if as defined, $f^{-1}_*\gamma_0$ was $-m\mu_0$ 
		or $-n\mu_0$, we can simply flip one pluming over before
		making the identification and redefine $f$ accordingly!
		\end{remark}
		Recall that $L(p^2,pq-1)$ admits a unique
		spin structure if $p$ is odd and two spin structures
		if $p$ is even.  In the former case, $f$ clearly maps
		the unique spin structure to itself.  In the later case,
		we investigate how $f$ behaves on spin structures by looking 
		at characteristic sublinks:  
		\begin{definition}[\cite{Kaplan}, Definition 1.10]
			For a framed link $L\subset S^3$, a sublink $L'\subset L$
			is characteristic if for each $K\subset L$,
			\[
				\lk(K,L') = \lk(K,K) \mod 2.
			\]
		\end{definition}
		When $M^3$ is given as (integral) surgery on $L$, spin structures on 
		$M$ are in bijection with characteristic sublinks of $L$.  	
		Furthermore, fixing a spin structure and thus a
		characteristic sublink of $M$, one can trace 
		where that structure goes under a diffeomorphism specified
		via handle moves / blow-ups by tracing how the sublink evolves
		under those moves (see \S5.7 of \cite{Gompf-and-Stipsicz}).	
		To accomplish this, we adopt the following
		notation to specify $(M,\mathfrak{s})$ 
		for $\mathfrak{s}\in \cal{S}(M)$ - the
		set of spin structures on $M$:
		\begin{notation}
			If $M^3$ is given by integral surgery on a framed link
			$L = K_1^{f_1}\cup \ldots \cup K_N^{f_N}$ with framings $f_i\in \Z$
			and $\mathfrak{s}\in \cal{S}(M)$ is a spin structure with associated 
			characteristic sublink  $L'\subset L$, then we denote
			\[
				(M,\mathfrak{s}) 
					= K_1^{(f_1;t_1)}\cup \ldots \cup K_{N}^{(f_N;t_N)}
			\]
			where each $t_i \in \Z/2\Z = \{1,-1\}$ 
			satisfies $t_i=-1$ if and only if $K_i\in L'$.				
		\end{notation}
		\noindent From \cite{Gompf-and-Stipsicz}, when sliding 
		$K_i$ over $K_j$, $(f_i;t_i)\mapsto (f_i+f_j\pm 2\lk(K_i,K_j);t_i)$ and
		$(f_j;t_j)\mapsto (f_j;t_it_j)$. Furthermore, 
		blowing-up corresponds to the addition of
		$(\pm1;-1)$-decorated unknot. From these two observations, we
		immediately conclude the following lemma.
		\begin{lem}\label{Lemma: Spinning}
			Suppose that a band of $k$ strands has $r$ strands contained
			in the characteristic sublink of a spin structure $\mathfrak{s}$ on $M$
			and the remaining $k-r$ strands 
			not in the characteristic sublink, then adding $-s_i$-full 
			twists to the band, through the introduction of a canceling pair,
			effects the characteristic sublink as in 
			Figure \ref{Figure: Characteristic Twisting)}
			\begin{figure}[!ht]
				\centering
					\includegraphics[scale=.6]{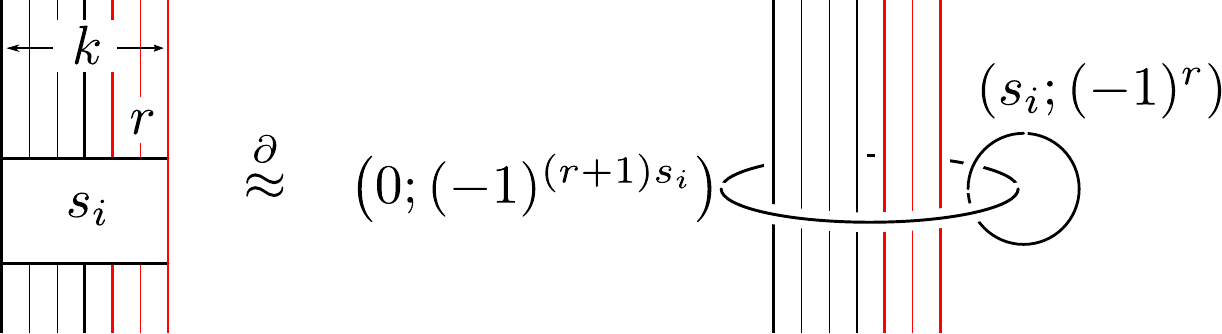}
				\caption{\small Tracing characteristic sublinks
						when introducing a canceling pair.}
						\label{Figure: Characteristic Twisting)}
			\end{figure}
			with no change to the original characteristic sublink and with
			framings within the band changing in the obvious way.
		\end{lem}
		\noindent
		Thus, we can refine Proposition \ref{Proposition: B_p,q Boundary Induction}
		to carry a fixed spin structure on $\partial B_{p,q}$ to each 
		$\partial B_{p,q}^i$.
		\begin{lem}
			Let $\mathfrak{s}\in \cal{S}(\partial B_{p,q})$ be specified by the 
			pair $(t_0,t_1)\in \Z/2\Z\times \Z/2\Z$, 
			then $\mathfrak{s}$ corresponds to
			the spin structure on $\partial B_{p,q}^i$ in 
			Figure \ref{Figure: B_p,q spin structures} where $T_0=t_0$ and
			for $1\leq i\leq \ell+1$,
			$T_i=(-1)^{1+\det A_{i-1}}(-t_0)^{\rho_{\ell+1-i}}(t_1)^{p\det A_{i-1}+ir_i}$	 
			such that $A_i$ and $\rho_{\ell+1-i}$
			are as defined in Lemma \ref{Lemma: Defining A(p-q,q)}.
		\begin{figure}[!ht]
			\centering
				\includegraphics[scale=.5]{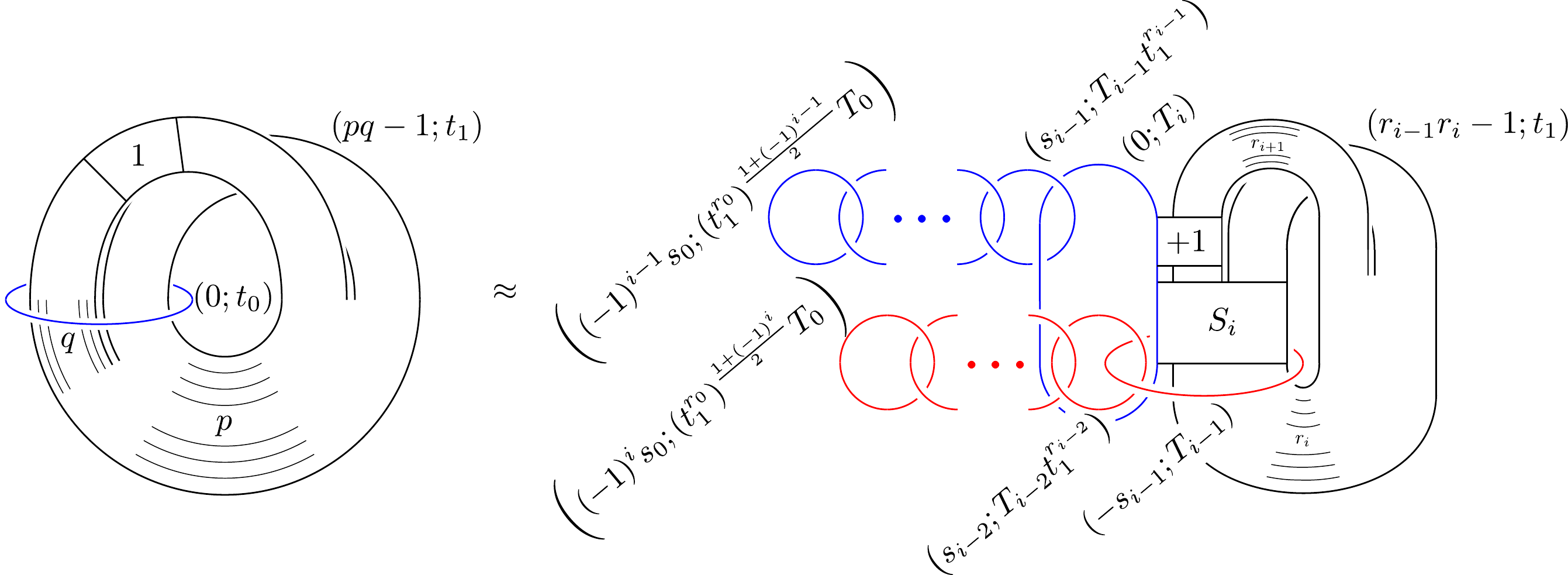}
			\caption{\small A fixed spin structure 
			on $\partial B_{p,q}$ and $\partial B_{p,q}^i$.}
			\label{Figure: B_p,q spin structures}
		\end{figure}
		\end{lem}
		\begin{proof}
			Starting with $(t_0,t_1)$ on $\partial B_{p,q}$, Lemma 
			\ref{Lemma: Spinning} combined with Proposition 
			\ref{Proposition: B_p,q Boundary Induction} gives 
			that the $T_j$'s in Figure \ref{Figure: B_p,q spin structures}
			are defined recursively by $T_{-1}\doteq 0$, $T_{0}\doteq t_0$,
			and $T_{j}=\left(-T_{j-1}t_1^{r_{j-1}}\right)^{s_{j-1}}T_{j-2}$.
			To see that the closed form for $T_j$ is as claimed, note that
			we can assume $T_j = (-1)^{a_j}(t_0)^{b_j}(t_1)^{c_j}$ for
			sequences $\{a_j\}, \{b_j\}, \{c_j\}\subset\Z$ which only need
			to be determined to their respective parities.  
			Then, the recursion on $T_j$ descends to 
			\[
				\begin{array}{ccc}
					a_{-1} \doteq 0		& b_{-1}\doteq 0 	& c_{-1}\doteq 0\\
					a_0 \doteq 0 & b_0\doteq 1 & c_0 \doteq 0 \\
					a_j=s_{j-1}(a_{j-1}+1)+a_{j-2}. 
						& b_j=b_{j-1}b_{j-1}+b_{j-2}. 
						& c_j=s_{j-1}(c_{j-1}+r_{j-1})+c_{j-2}.
				\end{array}
			\]
			By noting that $\rho_{\ell+1}=1$, $\rho_{\ell}=s_0$ and 
			$\rho_{\ell+1-j}=\rho_{\ell+1-(j-1)}s_{j-1}+\rho_{\ell+1-(j-2)}$
			the result follows by induction on $j$.
		\end{proof}
		\begin{remark}
			By Lemma \ref{Lemma: Defining A(p-q,q)}, we have that 
			$\det A_{\ell} = \pm d$ for $d$ defined therein.
			Thus, 
			\[
				T_{\ell+1} = 
						(-1)^{1+d}(-t_0)^{m}(t_1)^{p d+\ell+1}.
			\]
			If $p\in 2\Z$, then $t_1=-1$ for both spin structures on 
			$\partial B_{p,q}$ and we can further reduce $T_{\ell+1}$ 
			to $(-1)^{c+\ell} t_0$ 
			(as $m$ is necessarily odd and the parities of $c$ and $d$ 
			always oppose each other in this case). 
			Therefore, when $p\in2\Z$, 
			we can measure which spin structure $\mathfrak{s}$ gives on 
			$\partial B_{p,q}$ in the linear plumbing of Figure 
			\ref{Figure: B_p,q Boundary Induction - Final Case2)} by
			noting that the $-r_\ell$-framed unknot will be in the characteristic
			sublink associated to $\mathfrak{s}$ if and only 
			if $(-1)^{c+\ell} t_0=-1$.
			Of course, we can also measure this by looking at the $-s_0$-framed
			unlink.  However, to see which spin structure is induced on 
			$\partial A_{m,n}$, it is convenient to look at $-r_\ell$.
			To that end, we have
		\end{remark}
		\begin{proposition}\label{Proposition: Induced Spin Structures}
			Let $\mathfrak{s}$ be the spin structure on 
			$\partial B_{p,q}$ specified by $(t_0,t_1)$,   
			then $f_*(\mathfrak{s})$ is the spin structure on $\partial A_{m,n}$
			specified by 
			\[
				(v_0,v_1)
					=\left(\frac{(-1)^{c+\ell} t_0+t_1+(-1)^{c+\ell+1} t_0t_1+1}{2},t_1\right)
			\]  
			where the pair $(v_0,v_1)\in \Z/2\Z\times\Z/2\Z$ is analogously
			defined for $\partial A_{m,n}$ as the
			pair $(t_0,t_1)$ is in Figure \ref{Figure: B_p,q spin structures}
			for $\partial B_{p,q}$.
		\end{proposition}
	}
	\noindent 
	Using Proposition \ref{Proposition: Induced Spin Structures},
	we can deduce Corollary \ref{f extends} by carving.  
	Carving is a powerful tool for understanding handle decompositions
	(see, for instance, \cite{Akbulut-Carving} and \cite{Akbulut-book}). 
	The fact that carving $f(\mu_1)$ gives $S^1\times B^3$ is enough
	to extend $f$ to a diffeomorphism between $B_{p,q}$
	and $A_{m,n}$:	
	\begin{proof}[Proof of Corollary \ref{f extends}]
		By Theorem \ref{Theorem: Boundary Diffeomorphism},
		there exists $f:\partial B_{p,q} \to \partial A_{m,n}$
		satisfying that $f$ carries the belt sphere, 
		$\mu_1$, of the single 2-handle	in $B_{p,q}$ to 
		an unknot in $\partial A_{m,n}$.  Remarks
		\ref{Remark: B_p,q 0-framing preserved} and
		\ref{Remark: A_m,n 0-framing preserved} show that the 0-framing
		on $\mu_1$ determined by the cocore of the 2-handle is
		preserved as well.  Therefore, $f$ can be defined across the 
		cocore of the 2-handle in $B_{p,q}$.  Thus, we can view 
		$f$ as giving a diffeomorphism, $f_0$, between the result of
		0-surgery on $\mu_1\subset \partial B_{p,q}$ to that
		of $f(\mu_1)=\gamma\subset \partial A_{m,n}$.  As carving both 
		$\mu_1$ and $f(\mu_1)$ gives $S^1\times B^3$,
		$f_0$ is a diffeomorphism of $S^1\times S^2$ 
		to itself which extends uniquely 
		over $S^1\times B^3$ since we can verify
		that $f_0$ doesn't 
		intertwine the spin structures of $S^1\times S^2$
		by examining 
		Proposition \ref{Proposition: Induced Spin Structures}.
	\end{proof}
	
	\section{Homotopy Invariants of the Induced Contact Structures}
		\label{Section: Homotopy Invariants}
	{
	In this section, we compare the homotopy invariants of
	the contact structures induced by $\widetilde{J}_{m,n}$
	on $\partial A_{m,n}$ to those induced by the Stein 
	structures of $B_{p,q}$.  The latter are known to induce
	contact structures which are
	contactomorphic to the standard contact structure,
	$(L(p^2,pq-1),\bar{\xi}_{st})$ - thus Lisca's classification 
	result (Theorem \ref{Lens Space Symplectice Filling Classification})
	applies. For identifying tight contact structures on lens spaces,
	it is enough to know that the two contact structures in question are 
	homotopic up to contactomorphism.  Indeed, the following result of Honda's ensures this.	
	\begin{thm}[\cite{Honda-ClassificationI}, Proposition 4.24]
		\label{Tight + Homotopic = Isotopic}
		The homotopy classes of the tight contact structures of $L(p,q)$
		are all distinct.
	\end{thm}
	Further, it is known for contact structures with $c_1$ torsion 
	(which is always satisfied for 3-manifolds with $b_1=0$; 
	e.g. lens spaces) that particular homotopy
	invariants completely determine their homotopy classes.
	In \cite{Gompf-SteinHandles}, Gompf defines two invariants, $d_3$
	and $\Gamma$, and proves:
	\begin{thm}[\cite{Gompf-SteinHandles}, Theorem 4.16]
	\label{Theorem: Gompf's Homotopy Criteria}
		If $(M^3,\xi_i)$ for $i=1,2$, satisfies that $c_1(\xi_1)$
		is torsion and $\Gamma(\xi_1,\mathfrak{s})=\Gamma(\xi_2,\mathfrak{s})$
		for some spin structure $\mathfrak{s}$,
		then $\xi_1$ is homotopic to $\xi_2$ if and only if their 
		$d_3$ invariants coincide.		
	\end{thm}	
	We recall the definitions of $d_3$ and $\Gamma$.  For the
	three-dimensional invariant, $d_3$, we use the
	normalized definition found in \cite{Ozbagci-Stipsicz} -
	but note that it is equivalent to the definition of $\theta$ in
	\cite{Gompf-SteinHandles} which relies on the fact
	that each contact 3-manifold can be realized 
	as the boundary of an almost complex 4-manifold 
	as well as the fact that for $(X^4,J)$, 
	a closed almost complex 4-manifold, the quantity 
	$c_1^2(X,J)-3\sigma(X)-2\chi(X)=0$
	where $\sigma(X)$ and $\chi(X)$ are the signature 
	and Euler characteristic of
	$X$ respectively.  	
	\begin{definition}[\cite{Gompf-SteinHandles}, Definition 4.2]
		For a contact 3-manifold $(M,\xi)$ with $c_1(\xi)$ torsion,
		the three-dimensional invariant
		\[
			d_3(\xi) 
				= \frac{1}{4}\left(c_1^2(X,J)-3\sigma(X)-2\chi(X)\right)\in \Q
		\]
		for any almost complex 4-manifold $(X,J)$ with $\partial X=M$
		satisfying $J TM\cap TM=\xi$.
	\end{definition}
	$\Gamma$ associates to each spin structure on $(M,\xi)$
	an element of $H_1(M;\Z)$.  This is accomplished by noting
	that each spin structure on $(M^3,\xi)$ provides a trivialization 
	of $TM$, which, in turn, identifies $\mbox{Spin}^c(M)$
	with $H^2(M;\Z)$.  Then, with respect to this identification, 
	$\Gamma(\xi,\mathfrak{s})$ is \Poincare dual to 
	the spin$^c$-structure induced by $\xi$.	
	If $(M,\xi) = \partial (X,J)$, 
	a Stein domain, \cite{Gompf-SteinHandles} provides the following characterization
	of $\Gamma$ that we make extensive use of.
	Suppose that $(X,J)$ is obtained by attaching
	2-handles to a Legendrian link $K_1\cup\ldots \cup K_k$
	in $\partial(S^1\times B^3 \natural \ldots \natural S^1\times B^3)$
	with Seifert framings given by $\tb(K_i)-1$.
	Let $\tilde{X}$ be the result of surgering each one handle
	and let $L_0$ be the collection of 0-framed
	unknots, resulting from those surgeries.
	\begin{proposition}[\cite{Gompf-SteinHandles}, Theorem 4.12]
	\label{Proposition: Gamma Description}
		Let $(X,J)$ and $\tilde{X}$ be defined as above. Orient 
		$K_1\cup \ldots\cup K_k\cup L_0$ to obtain a spanning set
		for $H_2(\tilde{X};\Z)$.  Then 
		$\Gamma(\xi,\mathfrak{s})\in H_1(\partial X;\Z)$ is \Poincare
		dual to the restriction of the class $\rho \in H^2(X;\Z)$ whose
		value on each $[K_i]$ is given by
		\[
			\rho([K_i]) 
			=\frac{1}{2}\left(\mbox{rot}(K_i)+\lk(K_i,L'+L_0)\right)\in \Z
		\]
		where $\mbox{rot}(K)=0$ for each $K\in L_0$ and where $L'$ 
		is the characteristic sublink associated to $\mathfrak{s}$.
	\end{proposition}

	\begin{proposition}\label{Proposition: c_1 coming from B_p,q}
		For $p>q\geq 1$ relatively prime, the contact structure induced by the Stein
		structure, $J_{p,q}$, on $B_{p,q}$ given by Figure \ref{Figure: (B_p,q,J_p,q}
		\begin{figure}[!ht]
			\centering
				\includegraphics[scale=.45]{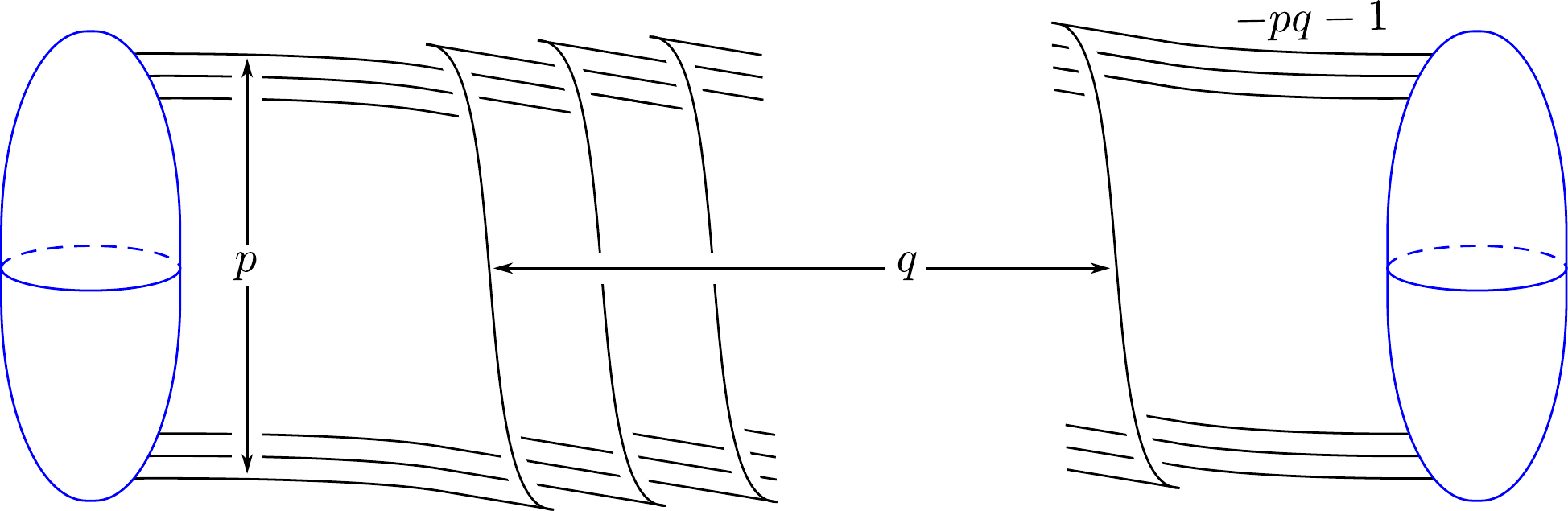}
			\caption{\small $(B_{p,q},J_{p,q})$}\label{Figure: (B_p,q,J_p,q}
		\end{figure}
		has $\Gamma(\xi_{J_{p,q}},\mathfrak{s}) = \frac{pq}{2} \cdot \mu_0$
		in an appropriate basis of $H_1(L(p^2,pq-1);\Z)$ and for a fixed choice of 
		$\mathfrak{s}$ when $p\in 2\Z$.
	\end{proposition}
	
	\begin{proof}
		Let $K_0$ be the boundary of the carving disk of
		the 1-handle in Figure \ref{Figure: (B_p,q,J_p,q}
		let $K_1$ be the attaching circle of the single 2-handle,
		and let $X_{p,q}$ be the 4-manifold obtained from Figure \ref{Figure: (B_p,q,J_p,q}
		by surgering the 1-handle (exchanging the ``dot'' on $K_0$ for a 0-framed 2-handle).  
		Then, let $\mathfrak{s}\in \cal{S}(\partial B_{p,q})$ be the spin structure on
		$\partial B_{p,q}$ specified by $(t_0,t_1)$ in Figure 
		\ref{Figure: B_p,q spin structures}.  As we have to slide the 2-handle under
		1-handle $q$-times to arrive at Figure \ref{Figure: (B_p,q,J_p,q}, 
		we see that $\mathfrak{s}$ corresponds 
		to the characteristic sublink 
		$$L' = \frac{1-t_0t_1^q}{2}K_0+\frac{1-t_1}{2}K_1$$
		in $X_{p,q}$.
		Orient the 2-handles so that $\mbox{rot}(K_1)=q$ and so that $\lk(K_0,K_1)=p$.
		In this orientation, let $\tilde\mu_i$ be a right handed meridian for 
		$K_i$ in $X_{p,q}$ and let $\mu_i$ be a right handed meridian for
		the corresponding (oriented) knots in $\partial B_{p,q}$ of 
		Figure \ref{Figure: B_p,q spin structures} so that
		\begin{align*}
			H_1(\partial X_{p,q};\Z)&=
				\left\langle\tilde{\mu}_0, \tilde{\mu}_1 : 
						p\tilde{\mu}_1=0, p\tilde{\mu_0}=(pq+1)\tilde{\mu}_1\right\rangle,\\
			H_1(\partial B_{p,q};\Z)&=
				\left\langle\mu_0, \mu_1 : p\mu_1=0, p\mu_0=(1-pq)\mu_1\right\rangle,
		\end{align*}
		where $\tilde{\mu}_0=\mu_0+q\mu_1$ and $\tilde{\mu}_1=\mu_1$.
		Then, for $j=0,1$, by Proposition \ref{Proposition: Gamma Description},\
		we have
		\[
			\rho([K_j]) = \frac{1}{2}\left(\frac{1-t_1}{2}p\right)(1-j)
			+\frac{1}{2}\left(q+\frac{3-t_0t_1^q}{2}p-\frac{1-t_1}{2}(pq+1)\right)j.
		\]
		Noting that $\mu_1=p\mu_0$, we find that
		\begin{align*}
			\Gamma(\xi_{J_{p,q}},\mathfrak{s}) 
				&= \frac{1}{2}\left(\frac{1-t_1}{2}p\right)\tilde\mu_0
			+\frac{1}{2}\left(q+\frac{3-t_0t_1^q}{2}p-\frac{1-t_1}{2}(pq+1)\right)\tilde\mu_1\\
			 &= \left(\frac{pq}{2}+\left(\frac{3-t_0t_1^q}{2}\right)\frac{p^2}{2}\right) \cdot \mu_0.
		\end{align*}
		Since there is no 2-torsion in $\Z/p^2\Z$ if $p\in 2\Z+1$, $p^2/2=0$ 
		in that case.
		If $p\in 2\Z$, then we can take $\mathfrak{s}$ corresponding to 
		$(t_0,t_1)=(1,-1)$.  In either case, (fixing the spin structure)
		we have $\Gamma(\xi_{J_{p,q}},\mathfrak{s})= \frac{pq}{2}\cdot\mu_0$.		
	\end{proof}
		
	\begin{proposition}\label{Proposition: c_1 coming from A_m,n}
		For $n>m\geq 1$ relatively prime, the contact structure induced by the Stein
		structure $(A_{m,n},\tilde{J}_{m,n})$ given by Figure \ref{Figure: (A_m,n,J_1)} has 
		\[
			\Gamma(\xi_{\tilde{J}_{m,n}},f_*(\mathfrak{s})) = 
			\frac{m+n}{2}\left((d-c)^2+\frac{1-t_1}{2}
						\left(1+(d-c)^2\left(mn
						+\frac{1+(-1)^{c+\ell} t_0}{2}(m+n)\right)\right)\right)\gamma_0
		\]
		in an appropriate basis of $H_1(\partial A_{m,n};\Z)$ where $cm+dn=1$.
	\end{proposition}
	
	\begin{proof}
		Let $\tilde{X}_{m,n}$ be the 
		4-manifold obtained from $A_{m,n}$ by surgering the 1-handle.  
		Let $f_*(\mathfrak{s})\in \cal{S}(\partial A_{m,n})$ be the spin 
		structure corresponding to the characteristic sublink
		$(t_0,t_1)$ in $\partial B_{p,q}$.
		From Proposition \ref{Proposition: Induced Spin Structures},
		we have that 
		$f_*(\mathfrak{s})=
		\left(\frac{(-1)^{c+\ell} t_0+t_1+(-1)^{c+\ell+1} t_0t_1+1}{2},t_1\right)$.
		Then, since we slide the 2-handle once under the 1-handle
		to get to Figure \ref{Figure: (A_m,n,J_1)}, we consider
		the characteristic sublink
		\[
			L' = \frac{1-t_1}{2}\left(\left(
				\frac{1+(-1)^{c+\ell}t_0}{2}\right)K_0+K_1\right)
		\]
		where $K_0$ is the 0-framed unkot arising from the surgery and 
		$K_1$ is the Legendrian attaching circle of the single 2-handle.
		Orient $K_0$ and $K_1$ so that $\mbox{rot}(K_1)=1$ and so 
		that $\lk(K_0,K_1)=m+n$.  With respect to this orientation,
		let $\gamma_i$ be a right-handed meridian for
		$K_i$ (viewed in $\partial A_{m,n}$ prior to the single handle slide).  
		Then,
		by Proposition \ref{Proposition: Gamma Description},
		\begin{align*}
			\Gamma(\xi_{\widetilde{J}_{m,n}},f_*(&\mathfrak{s})) 
				=  \frac{1-t_1}{2} \frac{m+n}{2}\gamma_0
						+\frac{1}{2}\left(1+\frac{1-t_1}{2}\left(mn
							+\frac{1+(-1)^{c+\ell} t_0}{2}(m+n)\right)\right)\gamma_1\\
		\end{align*}
		To see that $\Gamma(\partial A_{m,n},\mathfrak{s})$ is
		as claimed, note that 
		\[
			H_1(\partial A_{m,n};\Z)=\<\gamma_0,\gamma_1 : 
			(m+n)\gamma_1=0, mn\gamma_1 = -(m+n)\gamma_0\>.
		\]
		Combining this with the following observation;
		for $c$ and $d$ with $cm+dn=1$, 
		we necessarily have $c(m+n)+(d-c)n=1$ and $d(m+n)-(d-c)m=1$.
		Multiplying these two equations gives that
		\[
			-(d-c)^2\cdot mn + \left(cd(m+n)-c(d-c)m+d(d-c)n\right)\cdot (n+m) = 1.
		\]
		Thus,
		\begin{align*}
			\gamma_1 
				&= \gamma_1 - \left(cd(m+n)-c(d-c)m+d(d-c)n\right)\cdot (n+m)\gamma_1\\
				&= \left(1-\left(cd(m+n)-c(d-c)m+d(d-c)n\right)\cdot (n+m)\right)\gamma_1\\
				&= -(d-c)^2\cdot mn \cdot \gamma_1 \\
				&= (d-c)^2\cdot(m+n)\cdot \gamma_0.
		\end{align*}
		Exchanging $\gamma_1$ for $(d-c)^2 (m+n)\gamma_0$ in
		$\Gamma(\xi_{\widetilde{J}_{m,n}},f_*(\mathfrak{s}))$
		gives the result.
	\end{proof}
	
	\begin{cor}\label{xi_J is contactomorphic to xi_J_1}
		Suppose that $n>m\geq 1$ and $p-q>q\geq 1$ are each
		relatively prime such that $A(p-q,q)=(m,n)$ or $A(p-q,q)=(n,m)$, 
		then $\xi_{\tilde{J}_{m,n}}$
		is contactomorphic to  $\xi_{J_{p,q}}$.
	\end{cor}
	\begin{proof}
		We show that, after a suitable identification of $\partial A_{m,n}$
		and $\partial B_{p,q}$, the homotopy class of 
		$\xi_{\tilde{J}_{m,n}}$ corresponds
		with that of $\xi_{J_{p,q}}$.
		Both contact structures arise as complex tangencies 
		of the boundaries of Stein structures on rational 4-balls.  
		As such, 
		\[
			d_3(\xi_{\tilde{J}_{m,n}}) 
				= d_3(\xi_{J_{p,q}}) 
				=-\frac{1}{2}.
		\]
		Therefore, we only need to show that by applying
		$f^{-1}:\partial A_{m,n} \to \partial B_{p,q}$
		of Theorem \ref{Theorem: Boundary Diffeomorphism},
		$$\Gamma(f^{-1}_*\xi_{\tilde{J}_{m,n}},\mathfrak{s})
		=\Gamma(\xi_{J_{p,q}},\mathfrak{s})$$ for some spin structure
		$\mathfrak{s}\in \cal{S}(\partial B_{p,q})$.
		Now, by Proposition \ref{Proposition: c_1 coming from A_m,n}
		along with Remark \ref{Remark: Homology Pullback}
		and Lemma \ref{Lemma: Reducing Gamma mod p^2}
		we have
		\begin{align*}
			\Gamma(f^{-1}_*&\xi_{\tilde{J}_{m,n}},\mathfrak{s})
			 = f^{-1}_*\Gamma(\xi_{\tilde{J}_{m,n}},f_*(\mathfrak{s}))\\
			&= \frac{p}{2}\left((d-c)^2+\frac{1-t_1}{2}
						\left(1+(d-c)^2\left(mn
						+\frac{1+(-1)^{c+\ell} t_0}{2}p\right)\right)\right)f^{-1}_*(\gamma_0)\\
			&= \left\{ \begin{array}{ll}
			\frac{p}{2}\left((d-c)^2+\frac{1-t_1}{2}
						\left(1+(d-c)^2\left(mn
						+\frac{1+(-1)^c t_0}{2}p\right)\right)\right)
						n\mu_0	& \mbox{if $\ell \in 2\Z  $},\\
			\frac{p}{2}\left((d-c)^2+\frac{1-t_1}{2}
						\left(1+(d-c)^2\left(mn
						+\frac{1+(-1)^d t_0}{2}p\right)\right)\right)
						m\mu_0	& \mbox{if $\ell \in 2\Z+1$} \end{array}\right.\\
			&= \frac{pq}{2}\mu_0 = \Gamma(\xi_{J_{p,q}},\mathfrak{s})
		\end{align*}
		where the case when $\ell\in 2\Z+1$ follows from Lemma
		\ref{Lemma: Reducing Gamma mod p^2}	by symmetry.
		It follows from Theorem \ref{Theorem: Gompf's Homotopy Criteria}
		that $\xi_{J_{p,q}}$ and $f^{-1}_* \xi_{\tilde{J}_{m,n}}$ are in the same homotopy 
		class and thus, by Theorem \ref{Tight + Homotopic = Isotopic},
		isotopic.  Therefore $f^{-1}$ gives a contactomorphism from
		$(\partial A_{m,n}, \xi_{\tilde{J}_{m,n}})$ to 
		$(\partial B_{p,q},\xi_{J_{p,q}})$.
	\end{proof}
	This completes the proof of Theorem \ref{Each Amn is Stein}.  Although
	Lisca's result allows us to conclude that $A_{m,n}\approx B_{p,q}$
	whenever their boundaries coincide, it does not tell us anything about the 
	Stein structures $\tilde{J}_{m,n}$ versus $J_{p,q}$.  
	In \cite{Lekili-Bpq}, the authors note that it is unknown
	whether or not $B_{p,q}$ admits more than one Stein structure.
	Clearly, Theorem \ref{Each Amn is Stein} fails to answer that
	question; although, it does provide another candidate for study.
	}

	\section{The Algebraic Details}\label{Section: Algebraic Details}
	{			
	In this section we state the necessary algebra
	used in the proofs of Sections \ref{Section: Boundary Diffeomorphisms}
	and \ref{Section: Homotopy Invariants}. We start by 
	giving a definition of the function $A$ of \cite{Yamada-Amn} which
	associates the relatively prime 
	pair $(m,n)$ to a given relatively prime pair $(p-q,q)$. 
	Rather than relying on Yamada's original definition, we provide a 
	description of $A$ which dovetails with the boundary diffeomorphisms
	of Section \ref{Section: Boundary Diffeomorphisms}.  The 
	following lemma gives that definition and proves that
	it is equivalent to Yamada's original definition.	
\begin{lem}\label{Lemma: Defining A(p-q,q)}
	Let $p - q > q \geq 1$ be relatively prime, 
	and let $\{r_i\}_{i=-1}^{\ell+1}$ and
	$\{s_i\}_{i=0}^\ell$ be defined as in Definition 
	\ref{Definition: Euclidean Sequences}.
	Define sequences $\{\sigma_i\}_{i=0}^\ell$ and $\{\rho_i\}_{i=-1}^{\ell+1}$ 
	by
	$\sigma_0\doteq r_\ell-1$, $\sigma_{i} \doteq s_{\ell-i+1}$
	for  $i\in \{1,\ldots, \ell\}$.  Define $\rho_i$
	recursively by setting $\rho_{\ell+1}\doteq 1$, $\rho_\ell \doteq s_0$, 
	and defining 
	\[
		\rho_{i}=\rho_{i+1}\sigma_{i+1}+\rho_{i+2}.
	\]
	Set $m\doteq \rho_0$ and $n\doteq \rho_{-1}$. Then for $m$ and 
	$n$ as defined, we have
	\[
		A(p-q,q) = \left\{ \begin{array}{ll}
			(m,n)	& \mbox{if $\ell \in 2\Z  $},\\
			(n,m)	& \mbox{if $\ell \in 2\Z+1$}.
		\end{array}\right.
		(-1)^{\ell}(-c,d) = \left(|\det A_{\ell-1}|+(r_\ell-1)|\det A_\ell|,
					|\det A_{\ell}|\right)
	\]
	where $c$ and $d$ are the unique integers,
	with $0 < (-1)^{\ell+1} c,(-1)^{\ell}d < p$,
	satisfying $cm+dn=1$,
	and 
	\[
		A_i=	\begin{pmatrix} 
					s_1 & 1 	& 		 &    \\
					1	& -s_2	& 1		 &    \\
						& 1		& \ddots & 1  \\
						&		& 1		 & (-1)^{i+1}s_i
				\end{pmatrix}.
	\]
\end{lem}
\begin{proof}
	Recall the definition of $A(p-q,q)$, as well as the pair
	$(c,d)$ in \cite{Yamada-Amn}:
	Set $(a_0,b_0)\doteq(p-q,q)$, $(m_0,n_0)\doteq (1,1)$,
	$(c_0,d_0) = (0,1)$.
	If $a_i>b_i$,
	\[
		(a_{i+1},b_{i+1}) \doteq (a_i-b_i,b_i),
			\hspace{.25in}
		(m_{i+1},n_{i+1}) \doteq (m_i+n_i,n_i),
			\hspace{.25in}
		(c_{i+1},d_{i+1}) \doteq (c_i,d_i+c_i)
	\] 
	and if $a_i<b_i$,
	\[
		(a_{i+1},b_{i+1}) \doteq (a_i,b_i-a_i),
			\hspace{.25in}
		(m_{i+1},n_{i+1}) \doteq (m_i,n_i+m_i),
			\hspace{.25in}
		(c_{i+1},d_{i+1}) \doteq (c_i+d_i,d_i).
	\] 
	Then $A(p-q,q)\doteq(m_N,n_N)$ and $-c_N m_N +d_N n_N = 1$
	for $N$ such that $a_N=b_N=1$ - which 
	exists since $(p-q,q)=1$. Since $p-q>q$, there is a subsequence 
	$\{(a_{i_j},b_{i_j})\}_{j=1}^{\ell+2}\subset\{(a_i,b_i)\}_{i=0}^N$
	satisfying
	\[
		(a_{i_j},b_{i_j}) = \left\{
			\begin{array}{ll}
				(r_j,r_{j-1}), & \mbox{if $j\in 2\Z+1$}, \\
				(r_{j-1},r_j), & \mbox{if $j\in 2\Z$}
			\end{array}\right.
	\]
	for $j\in \{1,\ldots, \ell+1\}$, and 
	$i_{\ell+2}=N$. Furthermore, for these indicies, we have
	\[
		(m_{i_j},n_{i_j}) = \left\{
			\begin{array}{ll}
				(\rho_{\ell-j+1},\rho_{\ell-j+2}), & \mbox{if $j\in 2\Z+1$},\\
				(\rho_{\ell-j+2},\rho_{\ell-j+1}), & \mbox{if $j\in 2\Z$},
			\end{array}\right.
	\]
	Thus for $j=\ell+2$ we find that
	\[
		A(p-q,q) = (m_N,n_N) = \left\{
			\begin{array}{ll}
				(\rho_{-1},\rho_{0}), & \mbox{if $\ell \in 2\Z+1$},\\
				(\rho_{0},\rho_{-1}), & \mbox{if $\ell \in 2\Z$}.
			\end{array}\right.
	\]
	To see that this gives the claim for $(c,d)$ as well, we note
	for $j\leq\ell+1$, we have 
	\[
		(c_{i_j},d_{i_j}) = \left\{
			\begin{array}{ll}
				\left(|\det A_{j-2}|,|\det A_{j-1}|\right), 
					& \mbox{if $j\in 2\Z+1$},\\
				\left(|\det A_{j-1}|,|\det A_{j-2}|\right), 
					& \mbox{if $j\in 2\Z$}.
			\end{array}\right.
	\]
	where $A_{-1}\doteq 0$ and $A_0\doteq 1$.    
	Now, to produce such a subsequence, 
	take $i_1=s_0-1>1$ (so that $a_i>q$ for each $i< i_1$)
	similarly, take $i_{k+1} = s_{k}+i_{k}$ for 
	$k\leq \ell$ and take $i_{\ell+2} = i_{\ell+1}+r_\ell-1$.
	By definition,
	\[
		(a_{i_1},b_{i_1}) 
			= (p-q - (s_0-1)q , q) = (r_1,r_0).
	\]
	On the other hand
	\[
		(m_{i_1},n_{i_1})=(1+(s_0-1),1) = (\rho_\ell,\rho_{\ell+1}),
			\hspace{.25in}
		(c_{i_1},d_{i_1})=(0,1+0)=(0,1).
	\]
	For $i_{k+1}$ we have (for $k< \ell+1$),
	\begin{align*}
		(a_{i_{k+1}},b_{i_{k+1}}) 
			&= \left\{
				\begin{array}{ll}
					(r_k,r_{k-1}-s_kr_k), & \mbox{if $k\in 2\Z+1$} \\
					(r_{k-1}-s_kr_k,r_k), & \mbox{if $k\in 2\Z$}
				\end{array}\right.
			 = \left\{
				\begin{array}{ll}
					(r_k,r_{k+1}),	& \mbox{if $k+1\in 2\Z$} \\
					(r_{k+1},r_k), 	& \mbox{if $k+1\in 2\Z+1$}.\\
				\end{array}\right.
	\end{align*}
	and $(a_{i_{\ell+2}},b_{i_{\ell+2}})=(1,1)$. For $k\leq \ell+1$,
	\begin{align*}
		(m_{i_{k+1}},n_{i_{k+1}})
			&= \left\{
				\begin{array}{ll}
					(\rho_{\ell-k+1},\rho_{\ell-k+2}+s_k\rho_{\ell-k+1}), 
						& \mbox{if $k\in 2\Z+1$}\\
					(\rho_{\ell-k+2}+s_k\rho_{\ell-k+1},\rho_{\ell-k+1}), 
						& \mbox{if $k\in 2\Z$}
				\end{array}\right.\\
			&= \left\{
				\begin{array}{ll}
					(\rho_{\ell-k+1},\rho_{\ell-k+2}+\sigma_{\ell-k+1}\rho_{\ell-k+1}), 	
						& \mbox{if $k\in 2\Z+1$}\\
					(\rho_{\ell-k+2}+\sigma_{\ell-k+1}\rho_{\ell-k+1},\rho_{\ell-k+1}), 
						& \mbox{if $k\in 2\Z$}
				\end{array}\right.\\
			&= \left\{
				\begin{array}{ll}
					(\rho_{\ell-k+1},\rho_{\ell-k}),	
						& \mbox{if $k+1\in 2\Z$} \\
					(\rho_{\ell-k},\rho_{\ell-k+1}), 	
						& \mbox{if $k+1\in 2\Z+1$}.
				\end{array}\right.
	\end{align*}
	Finally notice that 
	\[
		\det A_i = (-1)^{i+1}s_i\det A_{i-1} - \det A_{i-2}
	\]
	and that the sign of $A_i$ coincides with 
	the sign of $\sin(\pi i/2)+\cos(\pi i/2)$ giving that 
	$|\det A_i|=s_i|A_{i-1}|+|A_{i-2}|$.  Therefore,	
	\begin{align*}
		(c_{i_{k+1}},d_{i_{k+1}})
			&= \left\{\begin{array}{ll}
				\left(|\det A_{k-2}|+s_k|\det A_{k-1}|,|\det A_{k-1}|\right), 
					& \mbox{if $k\in 2\Z+1$}\\
				\left(|\det A_{k-1}|,|\det A_{k-2}|+s_k|\det A_{k-1}|\right), 
					& \mbox{if $k\in 2\Z$}
			\end{array}\right.\\
			&= \left\{\begin{array}{ll}
				\left(|\det A_{k}|,|\det A_{k-1}|\right), 
					& \mbox{if $k+1\in 2\Z$}\\
				\left(|\det A_{k-1}|,|\det A_{k}|\right), 
					& \mbox{if $k+1\in 2\Z+1$}.
			\end{array}\right.
	\end{align*}
	When passing to $k=\ell+2$, we have
	\begin{align*}
		(c_{i_{\ell+2}},d_{i_{\ell+2}}) = \left\{
			\begin{array}{ll}
				\left(|\det A_{\ell}|,|\det A_{\ell-1}|+(r_\ell-1)|\det A_{\ell}|\right), 
					& \mbox{if $\ell\in 2\Z+1$},\\
				\left(|\det A_{\ell-1}|+(r_\ell-1)|\det A_{\ell}|,|\det A_{\ell-1}|\right), 
					& \mbox{if $j\in 2\Z$}.
			\end{array}\right.
	\end{align*}
	Giving that $(-1)^{\ell+1}\left(|\det A_{\ell-1}|+
		(r_\ell-1)|\right) m+(-1)^\ell |\det A_{\ell}|n=1$.
	\end{proof}
	In general, $c$ and $d$ satisfying $cm+dn=1$ are far from unique. 
	However, specifying them as in Lemma \ref{Lemma: Defining A(p-q,q)},
	(which are equivalent to the coefficients $s$ and $t$ that Yamada defines 
	in \cite{Yamada-Amn}) is crucial, since, as constructed, Yamada proves:
	\begin{lem}[\cite{Yamada-Amn},Lemma 2.5]\label{Lemma: Yamada's Lemma 2.5}
		Suppose that $A(p-q,q)=(m,n)$.
		 If $c$ and $d$ are defined as in Lemma \ref{Lemma: Defining A(p-q,q)},
		 giving that $cm+dn=1$, then $d-c=q$.
	\end{lem}
	Notice that if $A(p-q,q)=(n,m)$, then we clearly have $c-d=q$ instead.  
	Lemma \ref{Lemma: Yamada's Lemma 2.5} allows us to simplify
	the quantity $f_*^{-1}\Gamma(\xi_{\tilde{J}_{m,n}},f_*(\mathfrak{s}))$
	of Proposition \ref{Proposition: c_1 coming from A_m,n}.
	We only consider the case when $\ell \in 2\Z$ (giving that
	$A(p-q,q)=(m,n)$) since the case when $\ell\in 2\Z+1$ is
	symmetric by exchanging $m \leftrightarrow n$
	and $c\leftrightarrow d$.
	\begin{lem}\label{Lemma: Reducing Gamma mod p^2}
		Suppose that $A(p-q,q)=(m,n)$, and that $cm+dn=1$ 
		so that $d-c=q$, then for 
		$(t_0,t_1)\in \Z/2\Z\times\Z/2\Z$, 
		we have
		\[
			\frac{p}{2}\left(q^2+\frac{1-t_1}{2}
						\left(1+q^2\left(mn
						+\frac{1+(-1)^c t_0}{2}p\right)\right)\right)n  
			= \frac{pq}{2}
		\]
		in $\Z/p^2\Z$ whenever $p\in 2\Z+1$ or when $p\in 2\Z$ and
		$(t_0,t_1)=(1,-1)$.
	\end{lem}
	\begin{proof}
		Recall that $m+n=p$ and that 
		$qn=1-cp$. Thus, in $\Z/p^2\Z$
		\begin{align*}
			\frac{p}{2}
				&\left(q(1-cp)+\frac{1-t_1}{2}
					\left(
						n+m(1-cp)^2+\frac{1+(-1)^c t_0}{2} q(1-cp)p
					\right)
				\right)\\
			&= \frac{pq}{2}+\frac{p^2}{2}
				\left(
					-cq+\frac{1-t_1}{2}
					\left(
						1-2c+pc^2+\frac{1+(-1)^ct_0}{2}q
					\right)
				\right)\\
			&= \frac{pq}{2}+\frac{p^2}{2}
				\left(
					-cq+\frac{1-t_1}{2}
					\left(
						1+\frac{1+(-1)^c t_0}{2}q
					\right)
				\right).		
		\end{align*}
		If $p\in 2\Z+1$, then $\Z/p^2\Z$ lacks 2-torsion so that $p^2/2=0$.
		Suppose that $p\in 2\Z$ and that $(t_0,t_1)=(1,-1)$, then
		the above reduces to
		\begin{align*}
			\frac{pq}{2}+\frac{p^2}{2}\left(-cq+1
						+\frac{1+(-1)^c}{2}q \right)	=\frac{pq}{2}	
		\end{align*}
		since in this case, $q\in 2\Z+1$ and the quantity 
		$-cq+1+\frac{1+(-1)^c}{2}q$ is necessarily even.
	\end{proof}
	The following result is used to independently 
	verify that $\partial B_{p,q}\approx L(p^2,pq-1)$.
	To that end, we inductively build the linear plumbing
	of Figure \ref{Figure: B_p,q Boundary Induction - Final Case2)}
	from the middle out.  Furthermore, we choose
	signs on the weights so that 
	$-s_0$ ends up on the left. Since, a fortiori, we have
	\[
		[-s_0,s_1,\ldots,\pm r_\ell,1,\mp r_\ell,\ldots, -s_1,s_0] 
			=\frac{\det Q_{S_{\ell+1}}}{\det Q_{S_\ell^-}}
					= \frac{(-1)^\ell r_{-1}^2}{(-1)^{\ell}\left(1-r_{-1}r_{0}\right)} 
					= \frac{-p^2}{pq-1}
	\]
	where we use that if $[c_1,\ldots,c_n]=-p/q$ 
	then $-p/q=\det C_n/\det C_{n-1}$ for the matrices
	$C_i$ defined in Lemma \ref{Lemma: Moving Generators}.
	\begin{lem}\label{Lemma: intersection form algebra}
		Define $\{r_i\}_{i=-1}^{\ell+2}$ and $\{s_i\}_{i=0}^{\ell+1}$
		as in Definition
		\ref{Definition: Euclidean Sequences}, let $S_i$ be the 
		4-manifold given by plumbing $D^2$-bundles over $S^2$
		according to the weighted graph in Figure \ref{Figure: S_i}.
		\begin{figure}[!ht]
			\[
			\begin{tikzpicture}[scale=.7]
				\draw[thick] 
					(-2.95, 0) -- ( 2.95, 0)
					( 6.75,0)--( 4.95,0)
					(-6.75,0)--(-4.95,0); 
				
				\fill (-4.5,0) circle (1pt);
				\fill (-4  ,0) circle (1pt);
				\fill (-3.5,0) circle (1pt);						
				
				\fill ( 4.5,0) circle (1pt);
				\fill ( 4  ,0) circle (1pt);
				\fill ( 3.5,0) circle (1pt);						
				
				\fill ( 0   , 0) circle (3pt) node[above=3pt] { 1};
				\fill (-1.25, 0) circle (3pt) 
					node[above right=0pt, rotate=40] {$(-1)^{\ell}r_\ell$};
				\fill ( 1.25, 0) circle (3pt) 
					node[above right=0pt, rotate=40]  {$(-1)^{\ell+1}r_\ell$};
				\fill (-2.5,  0) circle (3pt) 
					node[above right=0pt, rotate=40] {$(-1)^{\ell-1}s_\ell$};
				\fill ( 2.5,  0) circle (3pt) 
					node[above right=0pt, rotate=40]  {$(-1)^\ell s_\ell$};
				\fill (-5.5,  0) circle (3pt) 
					node[above right=0pt, rotate=40] {$(-1)^{\ell-i+1}s_{\ell+2-i}$};
				\fill (-6.75, 0) circle (3pt) 
					node[above right=0pt, rotate=40] {$(-1)^{\ell-i}s_{\ell+1-i}$};
				\fill (5.5,   0) circle (3pt) 
					node[above right=0pt, rotate=40] {$(-1)^{\ell-i}s_{\ell+2-i}$};
				\fill (6.75,  0) circle (3pt) 
					node[above right=0pt, rotate=40] {$(-1)^{\ell+1-i}s_{\ell+1-i}$};
			\end{tikzpicture}
			\]
			\caption{\small The 4-manifold $S_i$.}
			\label{Figure: S_i}				
		\end{figure}
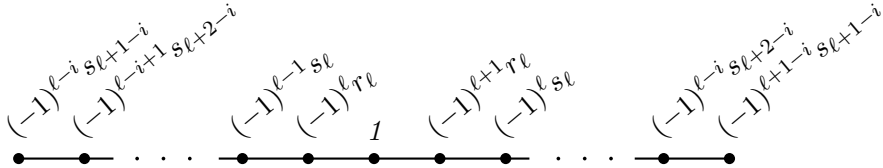
		Let $S_i^+$ be the 4-manifold obtained by plumbing an Euler class
		$(-1)^{\ell-i-1}s_{\ell-i}$ disk bundle to the Euler class
		$(-1)^{\ell-i}s_{\ell+1-i}$ disk bundle in $S_i$.
		Let $S_i^-$ be the 4-manifold obtained by plumbing an Euler class
		$(-1)^{\ell-i}s_{\ell-i}$ disk bundle to the Euler class
		$(-1)^{\ell+1-i}s_{\ell+1-i}$ disk bundle in $S_i$.  Then
		the intersection forms of $S_i$ and $S_i^\pm$ satisfy
		\begin{align*}
			\det Q_{S_i}	
				&=(-1)^{i+1} r_{\ell-i}^2,\\
			\det Q_{S_i^+} 	
				&= (-1)^\ell\left(r_{\ell-i-1}r_{\ell-i}+(-1)^{\ell+i}\right),\\
			\det Q_{S_i^-} 
				&= (-1)^{\ell}\left((-1)^{\ell+i}-r_{\ell-i-1}r_{\ell-i}\right).
		\end{align*}
	\end{lem}
}
\begin{proof}
	Induct on $i$ by noting that
	\[
		\det Q_{S_i^\pm}
			=(-1)^{\ell-i-(1\pm 1)/2}s_{\ell-i}\det Q_{S_i} 
				- \det Q_{S_{i-1}^\mp},
	\]
	\[
		\det Q_{S_{i+1}} 
			= (-1)^{\ell-i-1}s_{\ell-i}\det Q_{S_{i}^-}
				+(-1)^{\ell-i+1}s_{\ell-i}\det Q_{S_{i-1}^-} 
				+\det Q_{S_{i-1}},
	\]
	as well as the fact that, by definition, $r_{k}=r_{k+1}s_{k+1}+r_{k+2}$.
\end{proof}
	Finally, Lemma \ref{Lemma: Moving Generators}
	requires that we understand certain determinants
	arising from the intersection form of a given linear plumbing. 
	For the examples considered, we calculate those determinants here - they are used to
	express the generator, $\gamma_0$, of $H_1(\partial A_{m,n})$
	in terms of $\mu_0\in H_1(\partial B_{p,q})$.
	\begin{lem}\label{Lemma: Tracing gamma0}
		Let $\{\rho_i\}_{i=-1}^{\ell+2}$ and $\{\sigma_i\}_{i=0}^{\ell+1}$ 
		be as defined in Definition \ref{Definition: Euclidean Sequences},
		(associated to $n$ and $m$) then for each $i\leq \ell+1$ we have
		\[
			\det
			\begin{pmatrix}
				-\rho_\ell & 1 		&		&	\\
				1 	 	& \sigma_\ell 	& 1		&	\\
				  	 	& 1 		& \ddots& 1 \\
				  	 	&			& 1		& (-1)^{\ell+1-i}\sigma_{\ell+1-i}
			\end{pmatrix} 
				= -\left(\sin\left(\frac{\pi}{2}i\right)
						+\cos\left(\frac{\pi}{2}i\right)\right)\rho_{\ell-i}.
		\]			 
	\end{lem}
	\begin{proof}
		Induct on $i$, using that $\rho_{\ell+1}=1$ and that 
		$\rho_{\ell-i} = \rho_{\ell-i+1}\sigma_{\ell-i+1}+\rho_{\ell-i+2}$.
	\end{proof}
	\begin{cor}\label{Corollary: pulling back gamma0}
		Let $\gamma_0$, $\eta_{\pm1}$ each be meridians
		indicated in Figure \ref{Figure: A Basis for A_m,n}. 
		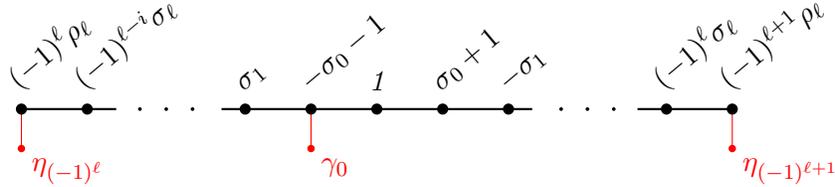
\begin{figure}[!ht]
			\[
			\begin{tikzpicture}[scale=.7]
				\draw[color=red] 
					( 6.75,0)--( 6.75,-.75)
					(-6.75,0)--(-6.75,-.75)
					(-1.25,0)--(-1.25,-.75);
				\fill[color=red] (-1.25, -.75) circle (2pt) 
					node[below right=0pt,color=red] {$\gamma_0$};
				\fill[color=red] (-6.75, -.75) circle (2pt) 
					node[below right=0pt,color=red] {$\eta_{(-1)^\ell}$};
				\fill[color=red] ( 6.75, -.75) circle (2pt) 
					node[below right=0pt,color=red] {$\eta_{(-1)^{\ell+1}}$};
				
				\draw[thick] 
					(-2.95, 0) -- ( 2.95, 0)
					( 6.75,0)--( 4.95,0)
					(-6.75,0)--(-4.95,0); 
				
				\fill (-4.5,0) circle (1pt);
				\fill (-4  ,0) circle (1pt);
				\fill (-3.5,0) circle (1pt);						
				
				\fill ( 4.5,0) circle (1pt);
				\fill ( 4  ,0) circle (1pt);
				\fill ( 3.5,0) circle (1pt);						
				
				\fill ( 0   , 0) circle (3pt) node[above=3pt] { 1};
				\fill (-1.25, 0) circle (3pt) 
					node[above right=0pt, rotate=40] {$-\sigma_0-1$};
				\fill ( 1.25, 0) circle (3pt) 
					node[above right=0pt, rotate=40]  {$\sigma_0+1$};
				\fill (-2.5,  0) circle (3pt) 
					node[above right=0pt, rotate=40] {$\sigma_1$};
				\fill ( 2.5,  0) circle (3pt) 
					node[above right=0pt, rotate=40]  {$-\sigma_1$};
				\fill (-5.5,  0) circle (3pt) 
					node[above right=0pt, rotate=40] {$(-1)^{\ell-i}\sigma_\ell$};
				\fill (-6.75, 0) circle (3pt) 
					node[above right=0pt, rotate=40] {$(-1)^{\ell}\rho_\ell$};
				\fill (5.5,   0) circle (3pt) 
					node[above right=0pt, rotate=40] {$(-1)^{\ell}\sigma_{\ell}$};
				\fill (6.75,  0) circle (3pt) 
					node[above right=0pt, rotate=40] {$(-1)^{\ell+1}\rho_{\ell}$};
			\end{tikzpicture}
			\]
			\caption{\small Expressing $\gamma_0$ in terms of a
			``preferred'' generator, $\eta_{-1}$, for the lens space $\partial A_{m,n}$.}
			\label{Figure: A Basis for A_m,n}				
		\end{figure}
		Then, fixing orientations so all linking is non-negative, we have
		\[
			-\left(\sin\left(\frac{\pi}{2}\ell\right)
						+\cos\left(\frac{\pi}{2}\ell\right)\right)m \cdot \eta_{(-1)^\ell} = 
			\gamma_0 = -\left(\sin\left(\frac{\pi}{2}\ell\right)
						+\cos\left(\frac{\pi}{2}\ell\right)\right)n \cdot\eta_{(-1)^{\ell+1}}.
		\]
	\end{cor}	
	\begin{proof}
		This follows immediately from Lemma \ref{Lemma: Moving Generators} and
		Lemma \ref{Lemma: Tracing gamma0}.		
	\end{proof}
		
	}
	\subsection*{Acknowledgments}
	{
		I would like to thank my advisor, Selman Akbulut,
		for making me aware of this problem - which initially
		came about through the solution to an exercise in
		\cite{Akbulut-book} and evolved into the questions arising from
		Yamada's work (\cite{Yamada-Amn, Yamada-LensSurgeries}) 
		addressed here.	
		I would also like to thank Christopher Hays and
		Faramarz Vafaee for many insights as this
		project has progressed.
	}
		
	\bibliographystyle{hep}
	\bibliography{References}

\begin{thebibliography}{{Yam}07}

\bibitem[Akb77]{Akbulut-Carving}
S.~Akbulut, \textsl{ On {$2$}-dimensional homology classes of {$4$}-manifolds},
\newblock Math. Proc. Cambridge Philos. Soc. \textbf{ 82}(1), 99--106 (1977).

\bibitem[Akb14]{Akbulut-book}
S.~Akbulut,
\newblock {$4$}-Manifolds,
\newblock \url{http://www.math.msu.edu/~akbulut/papers/akbulut.lec.pdf}, March
  2014.

\bibitem[CH81]{CassonHarer-RationalBalls}
A.~J. Casson and J.~L. Harer, \textsl{ Some homology lens spaces which bound
  rational homology balls},
\newblock Pacific J. Math. \textbf{ 96}(1), 23--36 (1981).

\bibitem[FS97]{FintushelStern-RationalBlowdown}
R.~Fintushel and R.~J. Stern, \textsl{ Rational blowdowns of smooth
  {$4$}-manifolds},
\newblock J. Differential Geom. \textbf{ 46}(2), 181--235 (1997).

\bibitem[Gom98]{Gompf-SteinHandles}
R.~E. Gompf, \textsl{ Handlebody construction of {S}tein surfaces},
\newblock Ann. of Math. (2) \textbf{ 148}(2), 619--693 (1998).

\bibitem[GS99]{Gompf-and-Stipsicz}
R.~E. Gompf and A.~I. Stipsicz,
\newblock \textsl{ {$4$}-manifolds and {K}irby calculus}, volume~20 of \textsl{
  Graduate Studies in Mathematics},
\newblock American Mathematical Society, Providence, RI, 1999.

\bibitem[Hon00]{Honda-ClassificationI}
K.~Honda, \textsl{ On the classification of tight contact structures. {I}},
\newblock Geom. Topol. \textbf{ 4}, 309--368 (2000).

\bibitem[Kap79]{Kaplan}
S.~J. Kaplan, \textsl{ Constructing framed {$4$}-manifolds with given almost
  framed boundaries},
\newblock Trans. Amer. Math. Soc. \textbf{ 254}, 237--263 (1979).

\bibitem[KY14]{Yamada-LensSurgeries}
T.~Kadokami and Y.~Yamada, \textsl{ Lens space surgeries along certain
  2-component links related with {P}ark's rational blow down, and
  {R}eidemeister-{T}uraev torsion},
\newblock J. Aust. Math. Soc. \textbf{ 96}(1), 78--126 (2014).

\bibitem[Lis08]{Lisca-LensFillings}
P.~Lisca, \textsl{ On symplectic fillings of lens spaces},
\newblock Trans. Amer. Math. Soc. \textbf{ 360}(2), 765--799 (electronic)
  (2008).

\bibitem[LM12]{Lekili-Bpq}
Y.~{Lekili} and M.~{Maydanskiy}, \textsl{ {The symplectic topology of some
  rational homology balls}},
\newblock ArXiv e-prints  (2012), {1202.5625}.

\bibitem[OS04]{Ozbagci-Stipsicz}
B.~Ozbagci and A.~I. Stipsicz,
\newblock \textsl{ Surgery on contact 3-manifolds and {S}tein surfaces},
  volume~13 of \textsl{ Bolyai Society Mathematical Studies},
\newblock Springer-Verlag, Berlin; J\'anos Bolyai Mathematical Society,
  Budapest, 2004.

\bibitem[Par97]{Park-GeneralizedRationalBlowdowns}
J.~Park, \textsl{ Seiberg-{W}itten invariants of generalised rational
  blow-downs},
\newblock Bull. Austral. Math. Soc. \textbf{ 56}(3), 363--384 (1997).

\bibitem[SS05]{StipsiczSzabo-ExoticCP2connect6CP2bar}
A.~I. Stipsicz and Z.~Szab{\'o}, \textsl{ An exotic smooth structure on {$\Bbb
  C\Bbb P\sp 2\#6\overline{\Bbb C\Bbb P\sp 2}$}},
\newblock Geom. Topol. \textbf{ 9}, 813--832 (electronic) (2005).

\bibitem[Sym98]{Symington-SymplecticRBD}
M.~Symington, \textsl{ Symplectic rational blowdowns},
\newblock J. Differential Geom. \textbf{ 50}(3), 505--518 (1998).

\bibitem[Sym01]{Symington-GeneralSymplecticRBD}
M.~Symington, \textsl{ Generalized symplectic rational blowdowns},
\newblock Algebr. Geom. Topol. \textbf{ 1}, 503--518 (electronic) (2001).

\bibitem[{Yam}07]{Yamada-Amn}
Y.~{Yamada}, \textsl{ {Generalized rational blow-down, torus knots, and
  Euclidean algorithm}},
\newblock ArXiv e-prints  (2007), {0708.2316}.

\end{thebibliography}

\end{document}